\documentclass[a4paper,reqno,11pt]{article}
\usepackage[a4paper,margin=3cm, marginparwidth=2cm]{geometry}
\usepackage[T1]{fontenc}
\usepackage{amsfonts,amsthm,amssymb,amsmath, mathdots, bbm,mathabx, bm}
\usepackage{amssymb, amscd,tensor}
\usepackage[makeroom]{cancel}
\usepackage[pdftex]{color,graphicx}
\usepackage[all,cmtip]{xy}
\usepackage{enumitem}
\usepackage{braket}

\numberwithin{equation}{section}

\newtheorem{Th}{Theorem}[section]

\newtheorem{Def}[Th]{Definition}
\newtheorem{Prop}[Th]{Proposition}

\newtheorem{Lem}[Th]{Lemma}

\newtheorem{Cor}[Th]{Corollary}

\newtheorem{Rk}[Th]{Remark}

\newtheorem*{theorem*}{Theorem}

\newcommand{\be}{\begin{equation}}
\newcommand{\ee}{\end{equation}}

\newcommand{\bE}{{\mathbb{E}}}

\newcommand{\cS}{{\mathcal{S}}}

\newcommand{\M}{{\sf M}}

\newcommand{\cP}{{\mathcal{P}}}

\newcommand{\Tr}{{\mathrm{Tr}}}

\newcommand{\bsig}{ {\bm{\sigma}}}
\newcommand{\btau}{{\bm{\tau}}}
\newcommand{\bnu}{{\bm{\nu}}}
\newcommand{\brho}{ {\bm{\hat\rho}} }

\newcommand{\bethat}{{\bm{\hat \eta}}}
\newcommand{\bhatmu}{{\bm{\hat \mu}}}

 \newcommand{\bS}{ { \bm{S} } }

\newcommand{\pC}{{p_C [\bsig, \btau,\ell ]}}

\newcommand{\id}{{\textrm{id}}}

\newcommand{\TitleGsig}{{\protect\boldmath${\sigma}$}}
\newcommand{\TitleGtau}{{\protect\boldmath${\tau}$$^{-1}$}}

\definecolor{armygreen}{rgb}{0.29, 0.33, 0.13}

\renewcommand\leq\leqslant

\renewcommand\geq\geqslant

\title{The tensor Harish-Chandra--Itzykson--Zuber integral I: Weingarten calculus and a generalization of monotone Hurwitz numbers}
\author{
Beno\^it Collins\thanks{Mathematics Department, Kyoto University, Kyoto, Japan.},
\; Razvan Gurau\thanks{Heidelberg University, Institut für Theoretische Physik, Philosophenweg 19, 69120 Heidelberg, Germany and CPHT, CNRS, Ecole Polytechnique, Institut Polytechnique de Paris, Route de Saclay, 91128 PALAISEAU, France},
\; Luca Lionni\thanks{IMAPP, Radboud University, Nijmegen, The Netherlands.}
}
\date{}

\begin{document}

\maketitle
\vspace{-0.2cm}
\begin{abstract}
We study a generalization of the Harish-Chandra--Itzykson--Zuber integral to tensors and its expansion over trace-invariants of the two external tensors. This gives rise to natural generalizations of monotone double Hurwitz numbers, which count certain families of constellations. We find an expression of these numbers in terms of monotone simple Hurwitz numbers, thereby also providing 
expressions for monotone double Hurwitz numbers of arbitrary genus in terms of the single ones. We give an interpretation of the different combinatorial quantities at play in terms of enumeration of nodal surfaces. In particular, our generalization of Hurwitz numbers is shown to enumerate certain isomorphism classes of branched coverings of a bouquet of $D$ 2-spheres that touch at one common non-branch node.
\end{abstract}

\tableofcontents

\section{Introduction}

\paragraph{The problem.} Our goal is to explore the following generalization of the HCIZ integral \cite{HarishChandra, Itzyk-Zub}:
\be 
\label{eq:HCIZTens1}
t\to I_{D,N}(t, A, B) =  \bE_U\Bigl[\exp \bigl(t \, \Tr (AUBU^* )\bigr)\Bigr]
\;. 
\ee
We will be particularly interested in the expansion of its logarithm, 
\[
{\cal C}_{D,N}(t, A, B)  = \log I_{D,N}(t, A, B),
\]
as a power series in $t$ and a Laurent series in $N$ and in its behavior at large $N$.
For $D=1$ we take $U$ a unitary $N\times N$ matrix $U \in U(N)$ ($U^*=U^{-1}$, where star denotes the adjoint), $A$ and $B$ self--adjoint $N\times N$ matrices and $\bE_U$ the expectation with respect to the Haar measure; that is for $D=1$ \eqref{eq:HCIZTens1} is the usual HCIZ integral \cite{HarishChandra, Itzyk-Zub}, here denoted by $I_{1,N}(t, A, B) $.

In this paper we are interested in the setup where $U$ is a tensor product of $D\ge 2$ unitary matrices $U=U^{(1)}\otimes\ldots \otimes U^{(D)}$ with $U^{(c)} \in U(N)$,
$\bE_U$ is the expectation with respect to the tensor product of $D$ Haar measures
and $A,B$ are self--adjoint $N^D\times N^D$ matrices called the external tensors.
In this case we will call \eqref{eq:HCIZTens1} the \emph{tensor HCIZ integral}. 
This is also written as:
\begin{equation}
\label{eq:HCIZTens2}
I_{D,N}(t, A, B) = e^{{\cal C}_{D,N}(t, A, B) } =  \int [dU] \; e^{t \Tr( A U BU^*) } \; .
\end{equation}

Let us note that if the tensors $A,B$ belong to the Lie algebra of $U(N)^{\otimes D}$ and are generic (i.e.~in the interior of a Weyl chamber), the integral $I_{D,N}(t, A, B)$
admits an exact determinantal formula as per Harish-Chandra's general results. In the particular case of $D=1$, this statement is important because 
it allows handling all multiplicity-free self-adjoint matrices $A,B$.
For $D\ge 2$ however, the Lie algebra is  smaller and the problem we are considering is much more general. 

\paragraph{Motivations.} 
The Kontsevich integral \cite{Kontsevich:1992ti} for a $N\times N$ self-adjoint matrix $X$:
\[
   \int dX \; e^{- \frac{1}{2}\Tr(XB X  ) + \imath  \frac{g}{3} \Tr(X^3) } \;, \qquad
   dX = \prod_{a,b}dX_{ab} \;,
\]
with $B$ a fixed $N\times N$ matrix is the prototype of a 
non-invariant probability distribution for a random matrix. 
This integral plays a crucial role in two dimensional quantum gravity \cite{DiFrancesco:1993cyw}. The Grosse--Wulkenhaar model \cite{Grosse:2003nw} is obtained by replacing the cubic potential with a quartic one and for a specific choice of $B$ this model is a $\phi^4$ quantum field theory on the non-commutative Moyal space \cite{Seiberg:1999vs} expressed in a matrix base \cite{Disertori:2006nq,Eichhorn:2013isa}. Such models can be generalized \cite{BenGeloun:2011rc,Carrozza:2013wda,Eichhorn:2018phj} to
rank $D$ complex tensors $T$ (with components denoted by 
$ T_{i^1\ldots i^D}$) transforming in the $D$--fundamental representation of the unitary group 
($T'_{i^1\dots i^D} = \sum_j  U^{(1)}_{i^1j^1} \ldots U^{(D)}_{i^D j^D} T_{j^1\ldots j^D}$). One is then interested in partition functions of the type:
\[
{\cal Z}  = \int dTd\bar T  \; e^{  -T B \bar T + V(T,\bar T)} \;, \qquad
T B \bar T =\sum_{i,j}\ T_{i^1\ldots i^D} B_{i^1\ldots i^D,j^1\ldots j^D }\overline{T_{j^1\ldots j^D}} \;,
\]
where $\bar T$ denotes the dual of $T$ (with complex conjugated entries $\overline{T_{i^1\ldots i^D}}$ and transforming in the conjugated representation) and $B$ some $N^D \times N^D$ matrix. The crucial point is that the perturbation $V(T,\bar T)$ is taken to be \emph{invariant} under the action of $U(N)^{\otimes D}$. 

A striking feature of both the Kontsevich integral and its generalizations involving random tensors is that one considers a non-invariant quadratic part and an invariant interaction. In order to study the interplay between these two, one can average over the unitary group:
\[
 {\cal Z}  = \int dTd\bar T  \; e^{ V(T,\bar T)} \int[dU] \; e^{  -T U B U^{\star}\bar T }  \; .
\]
The integral over $U$ is then just a particular case of \eqref{eq:HCIZTens2} for $A$ the tensor product of $T$ and its dual. 

More generally, the HCIZ integral is  extensively used in $D=1$ in random matrix models with non-invariant probability distributions, such as the two-matrix models \cite{Itzyk-Zub, Mehta, Eyn-two-mat},  and matrix models with an external source \cite{BR-Hik-1, BR-Hik-2, ZJ-ext-1, ZJ-ext-2, BR-Hik-3, Kuijlaars}, to cite just a few references. It is also central in studying the law of Gaussian Wishart matrices and non-centered Gaussian Wigner matrices \cite{Guionnet}. The study of the tensor HCIZ integral \eqref{eq:HCIZTens1} is justified by the generalization of these problems to random tensors. 
For example, it is natural in Quantum Information to study the sum of independent random tensors, so as soon as they have a $U(N)^{\otimes D}$-conjugation invariance, we expect that our results are a necessary preliminary towards to study of such asymptotic models. 

\

Another application of the HCIZ integral is as a generating function for the monotone double Hurwitz numbers \cite{Goulden1, Goulden2}. 
Hurwitz numbers count $n$-sheet coverings of the Riemann sphere by a Riemann surface of a certain fixed genus, where one branch point, for instance 0, is allowed arbitrary but fixed ramification, and all the other branch points are only allowed simple ramifications. Double Hurwitz numbers are such that not one but two points on the sphere, say 0 and $\infty$,  are allowed non-simple ramifications \cite{Goulden1, Goulden2, Goulden3, Goulden4}. 
Monotone single and double Hurwitz numbers are such that only a subset of the possible coverings are allowed. These numbers appear in $D=1$ as coefficients when expanding the HCIZ integral on the trace invariants of the two external matrices \cite{Goulden1, Goulden2}.

It is this last aspect on which the emphasis is put in the present paper: we expand the logarithm of the tensor HCIZ integral on the trace invariants of the two external tensors and study the $1/N$ expansion of the coefficients. This provides higher order generalizations of monotone double Hurwitz numbers. We provide a detailed study of the geometrical interpretation of these numbers and their relation to enumerations of branched coverings of a bouquet of $D$ spheres.

\paragraph{Final comments.}
Before proceeding let us comment some more on our model:
\begin{description}
\item{\emph{Different $N$'s.}} The generalization to the case of different dimensions $U^{(c)} \in U(N_c) $ is straightforward.
\item{ \emph{The D=1 case.}} 
For $D=1$ we get the HCIZ integral which is a Fourier transform 
of the $U(N)$-invariant probability measure concentrated on the orbit of the matrix $B$.
In general,
the same holds true, but for 
the 
smaller symmetry group $U(N)^{\otimes D}$ instead of $U(N^D)$. 
%
%
\item{\emph{Variants}.} Other models might be relevant, such as:
\[ t\to \log \bE_U\bigl[\exp (t \mathrm{Re} (\bar B UA )\bigr]\;, \] 
where $A,B$ are tensors. For $D=1$ this model boils down to the 
Br\'ezin-Gross-Witten (BGW) integral
\[ t\to \log \bE_{U}\bigl[\exp \bigl(t \Tr (AU+U^*B^* )\bigr)\bigr] \;. \]
which was largely studied in the literature.
For an optimal analytic result both for HCIZ and BGW in the
$D=1$ setup, we refer to \cite{Novak2006.04304}.
It turns out that the combinatorics are a bit different and slightly more involved, so we will move to this model in subsequent work. 
\item{\emph{Other groups.}} Similar results can be derived for the orthogonal or symplectic group. The initial theory does not change substantially, but the graphical interpretation does. We also keep this for future work. 
\end{description}

\paragraph{Plan of the paper.} 
The notations and prerequisite on Weingarten calculus, constellations and cumulants are gathered in Sec.~\ref{sec:prereq}, where the reader will also find, in 
Prop.~\ref{prop:moments-exp}, an expression for the moments of the tensor HCIZ integral which follows directly from the definitions.

\

The study of the cumulants of the tensor HCIZ integral is more involved. They write in terms of a \emph{cumulant Weingarten functions}, defined in Sec.~\ref{sub:exp-cumul-inv} and expressed in Sec.~\ref{sum:cum-wein-exact-1N} as series in powers of $1/N$  whose coefficients $p_C$ enumerate certain transitive factorizations of $D$-uplets of permutations. The rest of the paper is dedicated to the study and interpretation of the coefficients $p_C$.

\

Sec.~\ref{sub:asympt-cum-weing} contains our main theorem, Thm.~\ref{thm:asympt-cum-weing}. This theorem expresses
the coefficients $p_C$ as sums over partitions satisfying certain conditions. In this form we are able to compute $p_C$ at leading order in $1/N$, that is we identify the smallest exponent of $1/N$ with non vanishing contribution to $p_C$ and compute this contribution. 


\

In $D=1$, the coefficients $p_C$ are related to  \emph{monotone double Hurwitz numbers}, as detailed in Sec.~\ref{sub:Hurwitz}. These numbers are known to count certain isomorphism classes of connected branched coverings of the Riemann sphere. For $D>1$, the coefficients $p_C$  lead to a generalization of monotone double Hurwitz numbers, and one may wonder whether these numbers have a natural interpretation as enumerating certain branched coverings. 

\

This question is addressed in Sec.~\ref{sec:nodal}. After introducing nodal surfaces in Sec.~\ref{sub:def-nodal} we shown in Sec.~\ref{sub:Coverings} that the generalized Hurwitz numbers of Sec.~\ref{sub:Hurwitz} enumerates certain \emph{connected branched coverings of $D$ 2-spheres that ``touch'' at one common node (a bouquet of $D$ 2-spheres)}.  This provides (see Sec.~\ref{sub:Geom-descr-pC}) a geometric interpretation for the combinatorial formulas of Thm.~\ref{thm:asympt-cum-weing} 
and recasts the sums over partitions defining $p_C$ as a sum over certain nodal surfaces whose nodes are weighted with monotone single Hurwitz numbers. 

\newpage

\section{Prerequisite and direct results}
\label{sec:prereq}

\subsection{Notations}

Indices ranging from $1$ to $N$ will be denoted by $a_1,a_2,b_1,b_2$  and so on.
Let $S_n$ be the group of permutations of $n$ elements and $S_n^*$ the set of permutations different from the identity, $S_n^* = S_n\setminus \{\id\}$. For $\sigma\in S_n$,   $\#(\sigma)$ denotes the number of disjoint cycles of $\sigma$  and $\lVert \sigma \rVert$  the number of transpositions of $\sigma$ (i.e.~the minimal number of transpositions required to obtain $\sigma$)\footnote{The common notation would be $\lvert \sigma\rvert$, however we choose this notation instead to avoid confusion with the number of blocks of a partition.}.  These quantities satisfy the identity:
\be 
\label{eq:TransCyc}
\#(\sigma) + \lVert \sigma \rVert = n \;.
\ee
We denote by $(\rho_1,\dots, \rho_k) = (\rho_i)_{1\le i\le k }, \, \rho_i \in S_n$ or sometimes $\hat \rho$ an ordered sequence of $k$ permutations, that is a \emph{constellation} (see Sec.~\ref{sec:const}).

\

In this paper, we will deal with indices,  permutations, and sequences of permutations bearing a \emph{color} $c\in\{1,\dots, D\}$. The color is indicated in superscript or subscript: $a^c_1$ are indices, $\sigma_c$ are permutations, and so on. $D$-uplets will be written in bold, for instance $ \bsig = (\sigma_1,\ldots ,\sigma_D), \,\sigma_c \in S_n$ is a $D$-uple of permutations ($\bS_n$ denotes the set of $ (n!)^D$ such $D$--uplets), and $\brho$ is a $D$-uple of constellations. For $\bsig,\btau\in \bS_n$, we denote by $\bnu = \bsig\btau^{-1}$ the $D$-uple of permutations $\bnu = (\sigma_1 \tau_1^{-1}, \dots,  \sigma_D \tau_D^{-1})$.

\

We denote by $\pi,\pi',\pi_1,\pi_2$ and so on partitions of the set $\{1,\ldots , n\}$
and $\cP(n)$ the set of all such partitions. The notation $|\pi|$ is used for the number of blocks of $\pi$, while  $B\in \pi$ denotes the blocks, and $|B|$ the cardinal of the block $B$. $\le$ signifies the refinement partial 
order: $\pi'\le \pi$ if all the blocks of $\pi'$ are subsets of the blocks of $\pi$. Furthermore,  $\vee$ denotes the joining of partitions: $\pi\vee\pi'$ is the finest partition which is coarser than both $\pi$ and $\pi'$.
Let  $1_n$ be the one-block partition of $\{1,\dots, n\}$.

\

The partition induced by the transitivity classes of the permutation $\nu$ (i.e. the disjoint cycles of $\nu$) is denoted by $\Pi(\nu)$, hence $ |\Pi(\nu)| = \#(\nu)$. $d_p(\nu)$ denotes the number of cycles of $\nu$ with $p$ elements ($d_1(\nu)$ is the number of fixed points of $\nu$) and we have: 
\[ \sum_{p\ge 1} d_p(\nu) = \#(\nu) = |\Pi(\nu)| = n- \lVert \nu\rVert\; .\]

 Note that if $\pi \ge \Pi(\nu)$, then $\nu$ stabilizes the blocks of $\pi$, that is $\nu(B) = B$ for all $B\in \pi$.

\

Finally, $\Pi(\bsig,\btau)$ denotes the partition induced by the transitivity classes of the group generated by $\{\sigma_c, \tau_c \ |\ c\in\{1,\dots, D\}\} $, that is, $\Pi(\bsig,\btau) = \bigvee_{c=1}^D 
\big(\Pi(\sigma_c) \vee \Pi(\tau_c)\big)$ and $| \Pi(\bsig,\btau) |$ is its number of blocks. Note that all the permutations in $\bsig,\btau$ stabilize the blocks of some partition $\pi$ if and only if 
$ \Pi(\bsig,\btau) \le \pi$. 

\subsection{Trace invariants} 
\label{sub:trace-invariants}

We are interested in the invariants that can be built starting from  a $N^D \times N^D$ matrix  
$A$. We define the \emph{trace invariant} associated to $\bsig \in \bS_n$ as:
\[ \Tr_{\bsig}(A)=\sum_{\rm{all\ indices}} 
 \left( \prod_{s=1}^n A_{b^1_s\ldots b^D_s, a^1_s\ldots a^D_s } \right) 
\prod_{   c\in\{1,\ldots, D\}  } \left( \prod_{s=1}^n \delta_{a^c_s,b^c_{\sigma_c(s)}}  \right) \; . \]
These quantities are obviously invariant under conjugation by $U(N)^{\otimes D}$, that is $A\to UAU^{\star}$ with $U = U^{(1)} \otimes \ldots \otimes U^{(D)} , \, U^{(c)} \in U(N)$. 
For example:
\begin{itemize}
    \item[-] for $D=1$ any $\Tr_{\sigma}(A)$ is a product of traces of powers of $A$, 
    and the powers are the lengths of the cycles of $\sigma$:
    \be
     \Tr_{\sigma}(A) = \prod_{p\ge 1 } \big[ \Tr(A^{p} ) \big]^{d_p(\sigma)}  \;.
    \ee
    For $n=5, D=1$ and $\sigma = (123)(45)$ we get $\Tr_{\sigma}(A)= \Tr (A^2)  \; \Tr (A^3)$ .

\item[-] if all the $\sigma$'s are equal, $\sigma_c=\sigma$, then $\Tr_{\bsig}(A)$ is  again a product of traces of powers of $A$, but this time the traces are over indices of size $N^D$. Taking as before $\sigma = (123)(45)$ ($n=5$ and $D$ arbitrary) we get $\Tr_{(\sigma,\ldots , \sigma)}(A)= \Tr (A^2 ) \; \Tr (A^3) $.

\item[-] we finish by an example with different $\sigma$'s. For $n=2$, $D=2$, $\sigma_1=(12)$, and $\sigma_2=(1)(2)$:
$$\Tr_{(\sigma_1,\sigma_2)}(A)= \Tr_1 [ \Tr_2( A ) \Tr_2( A)] \;, $$
where $\Tr_{c}$ denotes the partial trace on the index of color $c$.
\end{itemize}

\subsection{Weingarten calculus} 

Weingarten calculus \cite{Weingarten} allows one to integrate any polynomial function on the unitary group.  
There exists a function $W^{(N)}: S_n\to \mathbb{R}$ such that, denoting $dU$ the Haar measure on $U(N)$, we have \cite{Collins03}:
\be
\label{eq:WeinDef}
\int dU \;   U_{i_1a_1}\cdots U_{i_na_n}\overline{U_{j_1b_1}\cdots U_{j_nb_n}}=
\sum_{\sigma,\tau\in S_n}
 \left(  \prod_{s=1}^n \delta_{i_s , j_{\sigma(s)}}  \right)
  \left( \prod_{s=1}^n \delta_{a_s , b_{\tau(s)}} \right)  \, W^{(N)} (\sigma\tau^{-1})  \; .
\ee
The function $W^{(N)}$ is uniquely defined if and only if $n\le N$, and
it follows from obvious commutativity relations that 
$W^{(N)}(\sigma\tau^{-1})$ depends only on the conjugacy class of $\sigma \tau^{-1}$, that is 
$W^{(N)}$ is a central function on the symmetric group $S_n$. The functions $ W^{(N)}$ are called Weingarten functions.

\paragraph{The $1/N$ expansion of the Weingarten functions.}
We start with a theorem that characterizes and defines the Weingarten functions.
Multiplying \eqref{eq:WeinDef} by $\prod_{s=1}^n\delta_{a_s,b_s} $ and summing the repeated indices, we get:
\begin{Th} [Collins-\'Sniady \cite{ColSni}]
The Weingarten function $\nu\to W^{(N)}(\nu)$ and the function
$ \nu \to \Phi(\nu) = N^{\# (\nu)} = N^{n} \delta_{\id; \nu}  + N^{n}  \sum_{\rho\neq \id} N^{-\lVert\rho\rVert}  \delta_{\rho;\nu} $ are pseudo-inverses for the convolution.
In particular, one has for $N\ge n$:
\[
\sum_{\nu \in S_n}  W^{(N)}(\nu) \, \Phi(\nu^{-1}\sigma)=\delta_{\sigma ; \id} \;.
\]
\end{Th}
This theorem can be used to compute the Weingarten functions:
\begin{equation}
\label{expansion-wg-permutations}
W^{(N)}(\nu) = N^{-n}\sum_{k\ge 0}\;\sum_{\substack{{\rho_1,\ldots , \rho_k\,\in\, S_n^*,}\\{ 
 \rho_1\cdots \rho_k =\nu}}}\;(-1)^k\,N^{-\sum_{i=1}^k\lVert\rho_i\rVert } \; ,
\end{equation}
with the convention that empty products are $1$ and empty sums are $0$. The 
case $k=0$ writes $\nu$ as an empty product in $S_n$, hence forces $\nu = \id$ and the empty sum 
$\sum_{i= 1}^k \lVert\rho_i\rVert$ is zero. This expansion is convergent for $N\ge n$.

The coefficient of $N^{-n-l}$ in the $1/N$ expansion of $W^{(N)}(\nu)$ is identified as \cite{Collins03}:
$$
W^{(N)}(\nu)= N^{-n}  \sum_{l \ge 0} (-1)^l  \, p(\nu ; l) \, N^{-l}   \;, \qquad 
 (-1)^l p(\nu ; l) =  \sum_{k\ge 0}\;(-1)^k\,m( \nu ; l, k), 
$$
where:
\be
\label{eq:def-weing-permutations}
m(\nu ; l, k) = \textrm{Card}\bigg\{\
 (\rho_i)_{1\le i\le k}  \; \bigg| \; \rho_i \in S_n^* \; , \text{ with } 
 \;
 \rho_1\cdots\rho_k =\nu \textrm{ and } \sum_{i=1}^k \lVert\rho_i\rVert = l\,\bigg\} \; ,
\ee
and for $l=0$ or 1, we have respectively $m(\nu;0,k) = \delta_{\nu,\id}\delta_{k,0}  $ and $m(\nu,1,k)= \delta_{k,1} \delta_{\lVert \nu \rVert ,1}$, and for $k=0$ or 1, we have
respectively $m( \nu ; l, 0)=  \delta_{\nu,\id}\delta_{l,0}$ and   $m(\nu,l,1) = \delta_{l, \lVert\nu\rVert} (1 - \delta_{\nu,\id})$. We conclude that $p(\nu; 0) = \delta_{\nu,\id}$ and $p(\nu,1) =  \delta_{\lVert \nu \rVert ,1} $. 

 This expression for the coefficient at order $N^{-n-l}$ as an alternating sum does not render explicit its sign. Another expression  \cite{collins-matsumoto} (see also \cite{matsumoto-novak}) solves this issue.
\begin{Def}
\label{def:weakly-monotone}
Let $(pq)$ be the elementary transposition of $p$ and $q$  (that is we use a cycle notation, but we omit the cycles with $1$ element).
An ordered $l$-tuple of transpositions  $\mu_1 = (p_1q_1) , \ldots, \mu_l =  (p_l q_l)$ is said to have  
\emph{weakly monotone maxima} if  $p_k<q_k$ for each $k\in \{1,\ldots, l\}$
and $q_k\le q_{k+1}$
for each $k\in \{1,\ldots, l-1\}$.
\end{Def}
\begin{Th} [\cite{collins-matsumoto,matsumoto-novak}]
We denote by $P(\nu ; l)$ the set of solutions of
$\nu = \mu_1  \ldots \mu_l$ with $\mu_1,\dots, \mu_l$ elementary transpositions with weakly monotone maxima. Then $p(\nu ; l)= | P(\nu ; l)|$.  
\end{Th}

\paragraph{Asymptotics of the Weingarten functions.}
  Classical theorems in combinatorics allow one to obtain the asymptotics of the Weingarten functions (Theorem 2.15 point $(ii)$ in \cite{Collins03}).

\begin{Cor}\label{cor:asympt}
For $\nu \in S_n$, we have the  asymptotic expansion:
\begin{equation} 
\label{eq:unitary-Wg-expansion2}
W^{(N)}(\nu)= N^{-n - \lVert\nu\rVert }  \, \M(\nu)  \, \big(1+O(N^{-2}) \big) \;,
\end{equation}
where $\M(\nu)$ is Biane-Speicher's M\"oebius function on the lattice of non-crossing partitions (\cite{NicaSpeicher}, Lecture 10) which is a central function which can be written in terms of the Catalan numbers:
\be
\label{eq:Moebius-on-NC}
\M(\nu) = \prod_{p\ge 1 }  \left[  \frac{ (-1)^{p-1} }{p} 
    \binom{ 2 p-2 }{ p-1  }  \right]^{d_p(\nu)} \; .
\ee
\end{Cor}

\subsection{Moments of the tensor HCIZ integral}
\label{sec:moments}
The moments of the tensor HCIZ integral \eqref{eq:HCIZTens1} write in terms of the Weingarten functions. 

\begin{Prop}
\label{prop:moments-exp} 
The moments of the tensor HCIZ integral \eqref{eq:HCIZTens1} are:
\begin{align}
\label{eq:moments-exp}
 \bE_U\bigg( [\Tr (AUBU^* )]^n\bigg) &=\int[dU] \; \big[\Tr (AUBU^* ) \big]^n  \crcr
& = \sum_{ \bsig,\btau \in \bS_n} 
\Tr_{\bsig}(A) \,
\Tr_{\btau^{-1}}(B)\, \prod_{c=1}^DW^{(N)}(\sigma_c\tau_c^{-1}) \; ,
\end{align}
where $\btau^{-1} = (\tau_1^{-1}, \ldots, \tau_D^{-1})$ and $W^{(N)}(\cdot )$ are the Weingarten functions.
\end{Prop}

\begin{proof} The proof is straightforward. Starting from:
\[
[\Tr (AUBU^* )]^n = \sum_{ \text{all indices} } 
\; \prod_{s=1}^n  A_{j_s^1 \dots j_s^D , i_s^1 \dots i_s^D } 
\left( \prod_{c=1}^D U^{(c)}_{i_s^ca_s^c} \right) B_{ a_s^1 \dots a_s^D , b_s^1 \dots b_s^D} \left( \prod_{c=1}^D \bar U^{(c)}_{j_s^c b_s^c} \right) \;,
\]
and using $D$ times the Weingarten formula  \eqref{eq:WeinDef}, the expectation amounts to:
\[
 \sum_{ \bsig,\btau \in \bS_n} \; 
 \sum_{ \rm{all\ indices} }  
\left(  \prod_{s=1}^n A_{j_s^1 \dots j_s^D , i_s^1 \dots i_s^D } B_{ a_s^1 \dots a_s^D , b_s^1 \dots b_s^D}  \right) \left( \prod_{c=1}^D \prod_{s=1}^n \delta_{i^c_s , j^c_{\sigma_c(s) } }   \,
\delta_{a^c_s , b^c_{\tau_c(s) } }   \right) \, W^{(N)}(\sigma_c\tau_c^{-1}) \;,
\] 
where we recognize the definition of the trace invariants (Sec.~\ref{sub:trace-invariants}). Observe that the first and the second index of the  $U$'s in \eqref{eq:WeinDef}
play slightly different roles, leading to the fact that the permutations for the invariant 
of $B$ are inverted. 

\end{proof}

From Corollary~\ref{cor:asympt}, we obtain the asymptotic expression of the moments:
\begin{equation}
\bE\bigg([ \Tr (AUBU^* ) ]^n\bigg)  
= \sum_{
\bsig, \btau \in \bS_n}  \,    \Tr_{\bsig}(A) \, \Tr_{\btau^{-1}}(B) \,
  \prod_{c=1}^D\M(\sigma_c\tau_c^{-1}) N^{-n - \lVert\sigma_c\tau_c^{-1}\rVert } 
  \bigl(1+O(N^{-2} ) \bigr) \;.
\end{equation}

\subsection{Constellations}\label{sec:const}

We now review some results on the enumeration of constellations. Constellations are central to the combinatorial interpretation of our main results.

\paragraph{Definition and graphical representation.} 

Intuitively, a combinatorial map (fatgraph, or ribbon graph in the physics literature) is a graph embedded in  
a closed surface\footnote{More precisely, the graph is drawn on the surface without edge-crossing and such that the complement of the graph in the surface is homeomorphic to a collection of discs.  The graph is then considered up to orientation preserving homeomorphisms of the surface whose restriction to the embedded graph is an isomorphism. } in which each edge is subdivided into two \emph{half}-edges. The map is bipartite if its vertices have one of two flavors (say $1$ and $2$) and every edge connects two vertices of different flavors.

Formally a bipartite combinatorial map, or a \emph{$2$-constellation}, is an ordered pair of permutations $\hat\rho = (\rho_1,\rho_2), \; \rho_1,\rho_2\in S_n$. It is represented canonically as an embedded graph as follows:
\begin{itemize}
 \item we let the \emph{flavor} $i=1,2$. For each cycle of $\rho_i$ we draw a vertex embedded in the plane (a disk). For each $s\in \{1, \dots, n\}$ we attach a \emph{half-edge}, \emph{i.e.} an outgoing segment to one of the vertices, labeled 
 $\braket{s}_i$. Every $s$ belongs to a cycle of $\rho_i$ and we draw the half-edges $\braket{s}_i, \braket{\rho_i(s)}_i, \braket{ \rho_i(\rho_i(s))}_i$ and so on ordered cyclically \emph{counterclockwise} around the vertex corresponding to this cycle. 
\item for every $s\in \{1,\dots, n\}$ we join the two half-edges $\braket{s}_1$ and $\braket{s}_2$ into an edge labeled $s$.
\end{itemize}

 The permutations $\rho_1$ and $\rho_2$ encode the ``successor'' half-edge: $\braket{\rho_i(s)}_i$ is the first half-edge encountered after $\braket{s}_i$ when turning counterclockwise around the vertex of flavor $i$ to which $\braket{s}_i$ belongs. 
The permutation $\rho_1\rho_2$ maps the edge $s$ onto the edge $\rho_1(\rho_2(s))$ obtained by first stepping from $s$ to $\rho_2(s)$, the successor of $s$ on the vertex of flavor $2$ to which $s$ is hooked, and then stepping from $\rho_2(s)$ to $\rho_1(\rho_2(s))$, the successor of $\rho_2(s)$ on the vertex of flavor $1$ to which $\rho_2(s)$ is hooked.
The cycles of $\rho_1 \rho_2$ are the \emph{faces} of the map.

\

\begin{figure}[!h]
\centering
\includegraphics[scale=1.1]{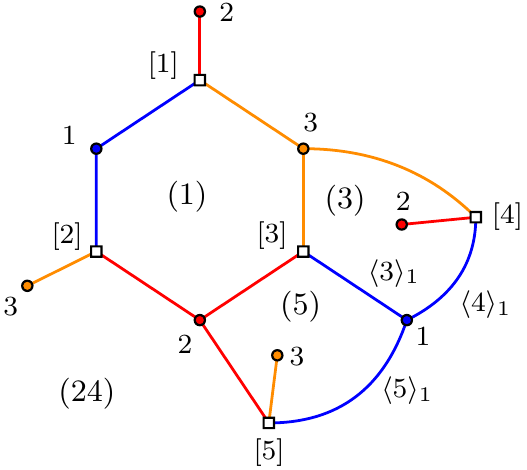}
\caption{A constellation with $k=3$ flavors and $n=5$ with $\rho_1=(12)(354)$,  $\rho_2=(1)(253)(4)$ and  $\rho_3=(134)(2)(5)$. Its faces are the cycles of $\rho_1\rho_2\rho_3=(1)(24)(3)(5)$, corresponding to three hexagons and one dodecagon. The labels in the figure are respectively the flavors of the vertices (no parenthesis) 
 the labels of some of the half edges (angled brackets),
the labels of the white vertices (square brackets), and the cycles of $\rho_1\rho_2\rho_3$ corresponding to the faces (parenthesis).
}
\label{fig:ex-constellation}
\end{figure}

This can be generalized to $k\ge 2$ flavors. A labeled \emph{$k$-constellation} is an ordered $k$-uple of permutations $ \hat{\rho} = (\rho_1,\ldots, \rho_k), \, \rho_1,\dots , \rho_k \in S_n$. The construction below is exemplified in Fig.~\ref{fig:ex-constellation}, which we will be using extensively:
\begin{itemize} 
\item To each permutation $\rho_i$, $i\in\{1,\ldots, k\}$ we associate a set of embedded \emph{vertices of flavor} $i$ corresponding to its cycles. In Fig.~\ref{fig:ex-constellation} the flavored vertices are represented as blue (for flavor $1$), red (for flavor $2$) and yellow (for flavor $3$). As $\rho_1 = (12)(354)$ there are two blue vertices in the figure, one bi-valent corresponding to the cycle $(12)$ and one tri-valent corresponding to the cycle $(354)$.

The vertices of flavor $i$ have a total of $n$ outgoing, cyclically ordered \emph{counterclockwise}, half-edges 
$\braket{s}_i$ for $s\in \{1,\dots, n\}$. 
For instance, in Fig.~\ref{fig:ex-constellation}, the three half-edges incident to the tri-valent blue vertex are labeled $\braket{ 3}_1,\braket{5}_1$ and $\braket{ 4}_1$.

\item To each $s\in\{1,\ldots,n\}$ we associate an embedded {\it white 
vertex} $[s]$.\footnote{The labeled white vertices generalize the labeled edges of the bipartite maps and can be viewed as hyper-edges.} We connect the half edges $\braket{s}_i$ for all the flavors $i$ to the vertex $[s]$ via edges such that the flavors are encountered in the order $1,\dots k$
when turning around the vertex $[s]$ \emph{clockwise}, see Fig.~\ref{fig:ex-constellation}. We label the edges by their end vertices as $([s],i)$ or $(i,[s])$.

\item The \emph{faces} of the constellation are the disjoint cycles of the product $\rho_1\cdots \rho_k$. A cycle of length $p$ corresponds to a face with $2pk$ corners (bounded by $2pk$ edges) that are alternatively white vertices $[s]$ and flavored vertices corresponding to $\rho_i$. The flavored vertices are encountered cyclically in the order $k, k-1,\dots 2, 1$ when going around the face while keeping the boundary edges to the left. 

The point is that the composition of the permutations $\rho_i$ encodes the walk around the faces of the constellation. To see this, let us consider the face to the right of the edge $([4],3)$ in Fig.~\ref{fig:ex-constellation}. Its perimeter consists in the edges 
\[ 
\begin{split}
&([4],3),(3,[1]),([1],2),(2,[1]),([1],1),(1,[2]), \crcr
&([2],3),(3,[2]),([2],2),(2,[5]),([5],1),(1,[4]) \;.
\end{split}
\]
The first two edges $([4],3),(3,[1])$ encode the fact that one passes from the white vertex $[4]$ to the white vertex $[1]$
by walking along a vertex of flavor $3$. This translates the fact that $\rho_3(4) = 1$. The next couple of edges $([1],2),(2,[1])$ translates the fact that $\rho_2(1) =1$, hence the first four edges together read in terms of permutations $\rho_2\rho_3(4) = 1$. Finally, $([1],1),(1,[2]) $ signifies that $\rho_1(1)=2$ which combined with the previous edges reads $\rho_1\rho_2\rho_3(4)=2$. Continuing the walk along the face encodes the action with $\rho_3,\rho_2$ and $\rho_1$ again. The face closes when we arrive back to the vertex $[4]$, at which point the tour around the face reads into $\rho_1\rho_2\rho_3\rho_1\rho_2\rho_3(4) = 4$.

\item The connected components of the resulting graph correspond to the transitivity classes of the group generated by $\{\rho_1,\dots, \rho_k\}$. 
Indeed for $s,s'\in \{1,\dots, n\}$ such that $\rho_{i_1} \cdots \rho_{i_l}(s) =s'$, the vertices $[s]$ and $[s']$ are connected in the constellation by the path: 
\[
([s],i_l)(i_l, [\rho_{i_l}(s)]) \dots ( [\rho_{i_2}\cdots \rho_{i_l}(s)]  i_1) ( i_1 [s']) \;. 
\]
The partition $\Pi(\rho_1,\dots, \rho_k) = \bigvee_{i=1}^k \Pi(\rho_i) \equiv\Pi(\hat\rho)$ is reconstructed by collecting all the white vertices $[s]$ belonging to the same connected component of the constellation into a block.
The number of connected components of a constellation is  $|\Pi(\hat \rho)|$. The constellation is said to be connected if $|\Pi(\hat \rho)|=1$, that is the group generated by $\{\rho_1,\ldots, \rho_k \} $  acts transitively on $\{1,\ldots,n\}$.
\end{itemize}

For $k=2$ the white vertices $[s]$ have valency 2 and therefore can be viewed as decorations (bearing labels) on edges, and we recover the bipartite maps described at the beginning of the section.

\paragraph{Euler characteristic.} A $k$-constellation $\hat \rho$ has $\sum_i \# ( \rho_i ) + n$ vertices,
$ kn $ edges, $\#(\rho_1\cdots  \rho_k)$ faces and $\lvert \Pi(\hat \rho)\rvert$ connected components. Being a combinatorial map, it has a non-negative genus, denoted by $g(\hat \rho)\ge 0$, and Euler characteristic:
\begin{equation}
\label{eq:Euler-constellations}
\sum_{i=1}^k \#(\rho_i) - n(k-1) + \#(\rho_1\cdots \rho_k) = 2 \lvert \Pi(\hat \rho)\rvert - 2g(\hat\rho) \;.
\end{equation}
A constellation (seen as a combinatorial map) is planar, $g(\hat\rho)=0$, 
 if and only if it can be drawn on the 2-sphere without edge-crossings such that each region of the complement of the graph on the sphere is homeomorphic to a disc. 
 
 Stated in terms of the length of the permutations, \eqref{eq:Euler-constellations} becomes:
\be
\label{eq:Euler-constellations-lengths}
\sum_{i=1}^k \lVert\rho_i\rVert =  2n -  \lVert \rho_1\cdots \rho_k\rVert + 2g(\hat \rho) - 2 \lvert \Pi(\hat \rho)\rvert \;.
\ee

Among the connected constellations, the planar ones are such that $\sum_{i=1}^k \lVert\rho_i\rVert$ is minimal at fixed $\nu = \rho_1\cdots \rho_k$.

\paragraph{Enumeration of planar constellations.} The main result we will need is due to \cite{BM-Schaeff} and concerns the enumeration of planar constellations. We fix $\nu\in S_n$. For $k\ge 2$, the number of \emph{connected planar} $k$-constellations $(\rho_1,\ldots, \rho_k)$ in $S_n$ with faces corresponding to disjoint
 cycles of
$\rho_1\cdots \rho_k=\nu$  is:
\be
\label{eq:BMSCH-mom-proper}
\tilde \gamma(\nu ; k)= k \frac{[(k-1)n - 1]!}{[(k-1)n - \#(\nu) + 2]!} \prod_{ p \ge 1} 
\Biggl[p\binom{k p -1}{p }\Biggr]^{d_p(\nu)} 
  =  \sum_{ \substack{  {\rho_1,\dots,\rho_k \in S_n, \, \rho_1\cdots \rho_k = \nu} \\
  { \Pi( \hat \rho  ) =1_n, \; \sum_i \lVert\rho_i\rVert = 2n -2 - \lVert\nu\rVert}
  } 
 } 1 
\; .
\ee
For the boundary values, we get:
\be
\tilde \gamma(\nu ; 0)= \delta_{\nu,\id} \;,\qquad \tilde \gamma(\nu ; 1)= \delta_{\#(\nu),1} \;.
\ee

This can be adapted for constellations $(\rho_1,\ldots, \rho_k)$ satisfying the same assumptions, but for which none of the permutations $\rho_i$ involved is the identity \cite{BM-Schaeff} (the constellations are said to be \emph{proper}), whose number is given by:
\be
\label{eq:planar-const-BMSCH}
\gamma(\nu ; k)= \sum_{j=0} ^{k} \binom{k}{j} \tilde \gamma(\nu ; j) (-1)^{k+j}  \; ,
\ee
and $\gamma(\nu ; 0) = \delta_{\nu;\id}$.
One can furthermore compute \cite{Collins03} the following alternating sum:
\begin{equation}
\label{eq:BMSCH-alt-sum}
\begin{split}
& \gamma(\nu) =  \sum_{k\ge 0 }\ (-1)^{k } \gamma( \nu, k) 
= (-1)^{\lVert\nu\rVert} \; \frac{(3n -\lVert\nu\rVert  - 3)!}{(2n )!} \prod_{p\ge 1} \Biggl [\frac {(2p)!}{p!(p-1)!}\Biggr]^{d_p(\nu)} \crcr
& \qquad \qquad = \M(\nu) \, \frac{(3n -\lVert\nu\rVert  - 3)!}{(2n )!}  \,
 \prod_{p\ge 1} \big[ 2p(2p-1) \big]^{d_p(\nu)}   \;,
\end{split}
\end{equation}
with $\M(\nu)$ the M\"oebius function on non-crossing partitions \eqref{eq:Moebius-on-NC}. 
Note that this is a class function.

More generally, we denote by $\tilde \gamma_l(\nu;   k )$ and $\gamma_l(\nu; k) $  the numbers of generic (\emph{i.e.} not necessarily proper) and respectively proper connected $k$-constellations with faces corresponding to the disjoint cycles of $\rho_1\cdots \rho_k =\nu$ and with $\sum_{i=1}^k\lVert\rho_i\rVert = l$ (hence genus $2g(\hat \rho)=l+2-2n + \lVert \nu \rVert$):
\begin{equation*}
\tilde \gamma_l(\nu ; k ) =
 \sum_{ \substack{  \rho_1,\dots,\rho_k \in S_n, \, \rho_1\cdots \rho_k = \nu \\ 
   \sum_i \lVert\rho_i\rVert = l , \; \Pi( \hat \rho  ) =1_n 
 } } 1   \;,\quad \; 
 \gamma_l(\nu ;  k ) = 
  \sum_{ \substack{  \rho_1,\dots,\rho_k \in S_n^*, \, \rho_1\cdots \rho_k = \nu \\ 
   \sum_i \lVert\rho_i\rVert = l , \; \Pi( \hat \rho  ) =1_n 
 } } 1 = \sum_{j=0}^k \binom{k}{j} \tilde \gamma_l(\nu;j) (-1)^{k+j} \;,
\end{equation*}
where the last equality follows by inverting the relation
$ \tilde \gamma_l(\nu ;  k )  = \sum_{j=0}^k  \binom{k}{j}\gamma_l(\nu ;  j )  $.

Finally, we denote by $\gamma_l(\nu)$ the alternating sum of the numbers of connected proper constellations with faces  corresponding to the disjoint cycles of $\rho_1\cdots \rho_k =\nu$ and $\sum_{i=1}^k\lVert\rho_i\rVert = l$:
\be
\label{eq:gamma_l}
 \gamma_l(\nu)  = \sum_{k\ge 0} (-1)^k \gamma_l(\nu;  k ) \;.
\ee
Eq.~\eqref{eq:BMSCH-mom-proper}, \eqref{eq:planar-const-BMSCH}, and \eqref{eq:BMSCH-alt-sum}  correspond to the minimal possible value
$l= 2n-2-\lVert\nu\rVert$.

\subsection{Cumulants}

For $X$ some random variable, the cumulant $C(X^n)$, also sometimes called connected correlation, is defined by:  
\[ {\cal C}(t) = \log \bE (\exp tX)= \sum_{p\ge 1}t^n\frac{C(X^n)}{n!} \; .\]
For instance, the second cumulant 
$ C(X^2) = \bE(X^2) -  \bE(X)^2 $
is the variance of the probability distribution of $X$.
The cumulants write in terms of the moments of the distribution and vice versa. In order to write down the relation between the two in a convenient form, we introduce some notation.

It is convenient to distinguish between the different factors $X$ in the monomial $X^n$. We do this by introducing a fictitious label $i=1,\dots n$ and writing $X^n = X_1 \cdots X_n$ where $X_i = X$ for all $i$. Then the $n$'th cumulant can be written as  
$C(X^n) \equiv C( \prod_{i=1}^n X_i) $.

For any partition $\pi\in \cP(n)$ with blocks $B\in \pi$, we define $C_{\pi}= \prod_{B\in\pi }  
C \left(\prod_{i\in B} X_i \right) $. We are now in the position to write the expectations in term of the cumulants:
\begin{equation}\label{eq:expect-cum}
 \bE \left( \prod_{i=1}^n X_i \right) =\sum_{\pi \in \cP(n) } C_{\pi } \;.
 \end{equation}

The equation \eqref{eq:expect-cum} can be inverted through the M\"oebius inversion formula
to yield the cumulants in term of the expectations. Defining
$
\bE_\pi = \prod_{B \in \pi } \bE \left( \prod_{i \in B} X_i  \right)$, we have:
\begin{equation}
\label{eq:mom-cum}
\bE_\pi= \sum_{\pi', \, \pi'\le \pi} C_{\pi '} \;,\qquad 
C_\pi= \sum_{\pi', \, \pi'\le \pi}   \lambda_{\pi',\pi}\bE_{\pi '} \;,
\end{equation}
where  $\lambda_{\pi',\pi}$ is the M\"oebius function for the lattice of partitions \cite{Rot64}\footnote{This should not be confused with the M\"oebius function for the lattice of \emph{non-crossing} partitions \eqref{eq:Moebius-on-NC}.}:
\[
\lambda_{\pi',\pi} = \prod_{B\in \pi} 
  (-1)^{|\pi'_{|B}|-1} \big(|\pi'_{|B}|-1 \big) ! \;,
\]
where $\pi'_{|B}$ is the restriction of the partition $\pi'$ to the block $B\in \pi$. This restriction is well defined because $\pi'\le \pi$. 
In particular, recalling that $1_n$ denotes the one block partition, we have:
\be
\label{eq:def-cum-Epi}
 C(X^n) = C_{1_n}   = \sum_\pi\lambda_{\pi}\bE_\pi \;  , \qquad \lambda_{\pi} \equiv \lambda_{\pi, 1_n}
   = (-1)^{|\pi|-1} (|\pi|-1)!
 \; .
\ee

\newpage

\section{Cumulants of the tensor HCIZ integral}
\label{sec:cumul}

The study of the cumulants of the tensor HCIZ integral is the core of this paper. They expand in terms of trace invariants of $A$ and $B$ times \emph{cumulant Weingarten functions} defined in Sec.~\ref{sub:exp-cumul-inv}. An expression of the latter as a series in $1/N$ is derived in Sec.~\ref{sum:cum-wein-exact-1N}. The coefficients of this expansion  are shown to count certain transitive factorizations of $D$-uplets of permutations. 

\subsection{The cumulant Weingarten functions}
\label{sub:exp-cumul-inv}

We denote $\chi(\cdot)$ the indicator function  which is one if the condition $\cdot$ is true and zero otherwise.

\begin{Def}[The cumulant Weingarten functions]
\label{def-of-cum-wein}
 For any partition $\pi$, let $W_\pi^{(N)}[\bsig, \btau ]$ be:
\begin{itemize}
\item[-] zero if at least one of the permutations involved in $\bsig$ or $\btau$ does not stabilize the blocks of $\pi$, that is $W^{(N)}_{\pi}[\bsig,\btau]$ is zero unless  $\Pi(\bsig,\btau) \le \pi $.

\item[-] the product over the blocks of $\pi$ of Weingarten functions involving permutations restricted to these blocks if all the permutations in $\bsig,\btau$ stabilize the blocks of $\pi$.
\end{itemize}

Denoting $\sigma_{c|B} $ the restriction of $\sigma_c$ to the block $B\in \pi$ (which is well-defined whenever $\sigma_c$ stabilizes the blocks of $\pi$), we have:
\[
 W_\pi^{(N)}[\bsig, \btau ] = \chi\big( \Pi(\bsig,\btau)\le \pi \big) \prod_{B \in \pi}\,  \prod_{c=1}^D W^{(N)}(\sigma_{ c |B} \tau_{c | B }^{-1} )
  \;.
\]

The \emph{cumulant Weingarten  function} $W_C^{(N)}[\bsig, \btau]$ is:
\be
\label{eq:WC-from-Wpi}
W_C^{(N)}[\bsig, \btau]=\sum_\pi \lambda_{\pi}W_\pi^{(N)}[\bsig ,\btau ] \;,
\ee
where $\lambda_{\pi} = \lambda_{\pi,1_n}$ is the M\"oebius function with the second argument set to the one-block partition.
\label{def:cum-weing-fun}
\end{Def}

Observe that, due to the indicator function, both $ W_\pi^{(N)}[\bsig, \btau ] $ and $ W_C^{(N)}[\bsig, \btau] $ depend on $\bsig$ and $\btau$ and not only on the product $\bsig\btau^{-1}$. The cumulant Weingarten functions arise naturally in the expansion of the cumulants of the tensor HCIZ integral over trace invariants.

\begin{Prop}
\label{prop:cumulants-first-expr}
The cumulants of the tensor HCIZ integral \eqref{eq:HCIZTens1} are:
\be 
\label{eq:cumulants}
C\bigg( [\Tr (AUBU^* )]^n\bigg)=\sum_{ \bsig,\btau\, \in\, \bS_{n}}
    \Tr_{\bsig}(A) \, \Tr_{\btau^{-1}}(B) \,   W_C^{(N)}[\bsig,\btau ] \;,
\ee
where $W_C^{(N)}[\bsig, \btau ]$ is the cumulant Weingarten functions, uniquely defined for $N\ge n$. 
\end{Prop}
\begin{proof} To any partition $\pi$, we associate the expectation:
\[
\bE_\pi =\prod_{B \in \pi} \bE\bigg(\prod_{i\in B}   \Tr( A U_i B U_i^* ) \bigg) 
 =  \sum_{ \bsig,\btau\,\in \bS_{n} } \Tr_{\bsig}( A) \, \Tr_{ \btau^{-1}}(B) \,   W_\pi^{(N)}[\bsig ,\btau ]
\; ,
\]
where the second equality follows from Prop.~\ref{prop:moments-exp}. It then follows from \eqref{eq:def-cum-Epi} that:
\[
C\bigg( [\Tr (AUBU^* )]^n\bigg)=\sum_{ \bsig,\btau\,\in \bS_{n} }
   \Tr_{\bsig}( A ) \, \Tr_{\btau^{-1}}( B ) \,   \sum_\pi \lambda_{\pi}W_\pi^{(N)}[\bsig ,\btau ] \; ,
\] 
which leads to Eq.~\eqref{eq:cumulants} using the definition \eqref{eq:WC-from-Wpi}.

\end{proof}

\subsection{Exact expression of the cumulant Weingarten functions}
\label{sum:cum-wein-exact-1N}

\begin{Th} 
\label{thm:1Nexpansion-Weingarten-Cumulants}
The cumulant Weingarten functions are:
\[
W_C^{(N)}[\bsig, \btau ]= N^{-nD}  \sum_{l \ge 0} \;(-1)^l
p_C [\bsig , \btau ;l]  \; N^{-l } \; , 
\]
where:
\[
(-1)^l p_C [\bsig , \btau ;l] =  \sum_{k\ge 0}\;(-1)^{k}\,m_C( \bsig , \btau ; l, k) \; , 
\]
and $m_C( \bsig , \btau ; l, k)$ is the number of $D$-uplets of constellations
$  ( \rho^c_{i_c} )_{1\le i_c\le k_c} , \, c\in\{1,\dots, D\} $, with the following properties: 
\begin{itemize}
\item all the permutations $\rho^c_{i_c}$ are different from the identity permutation,
\item for all $c\in\{1,\ldots,D\}$, $\sigma_c\tau_c^{-1}=\rho^c_{1}\cdots \rho^c_{k_c} $,
\item $\sum_{c=1}^Dk_c =k$ with $k_c\ge 0$, and $k_{c}=0$ implies $\sigma_{c}=\tau_{c}$, 
\item $\sum_{c=1}^D\sum_{i_c=1}^{k_c} \lVert\rho^c_{i_c}\rVert =l$, 
\item the collection of all $\big\{ \sigma_c,\tau_c, (\rho^c_{i_c} )_{1\le i_c\le k_c} \big \}_{1\le c\le D}$, acts transitively on $\{1,\ldots , n\}$. 
\end{itemize}
This expansion is convergent for $N\ge n$.
\end{Th}

The boundary values are $m_C( \bsig , \btau ; 0 , k )= \delta_{\lvert \Pi(\bsig,\btau)\rvert, 1} \delta_{\bsig; \btau}  \delta_{k,0}$ and $p_C[\bsig, \btau ; 0 ] = \delta_{\lvert \Pi(\bsig,\btau)\rvert, 1} \delta_{\bsig; \btau} $.

\begin{proof}
The functions $W_\pi^{(N)}[\bsig ,\btau ]$ in Def.~\ref{def-of-cum-wein} are non-trivial only if $\bsig$ and $\btau$ stabilize the blocks of $\pi$. We denote by $\bnu = \bsig\btau^{-1}$ and $\nu_{c|B} $ the restriction of $\nu_c$ to the block $B$. \eqref{expansion-wg-permutations} leads to:
\begin{align}
& W_\pi^{(N)}[\bsig ,\btau ] =  \chi\big( \Pi(\bsig,\btau)\le \pi \big) \prod_{c=1}^D\prod_{B \in \pi } N^{-\lvert B \rvert} \Bigl( \sum_{k_c^B  \ge 0} 
\sum_{ \substack{ \rho^{c,B}_{1} ,\ldots  \rho^{c,B}_{k_c^B }\, \in\, S_{|B|}^*  \\
  \rho^{c,B}_{1}  \cdots  \rho^{c,B}_{k_c^B } = \nu_{c|B} 
} } \;(-1)^{k_c^B} \,N^{-\sum_{i_c^B=1}^{k_c^B } \lVert\rho^{c,B}_{i_c^B}\rVert}
\Bigr) \crcr
& \qquad = \chi\big( \Pi(\bsig,\btau)\le \pi \big) \sum_{ \{k_c^B \}_{c,B} \ge 0} \;
\sum_{ \substack{ \big\{ \rho^{c,B}_{1} ,\ldots  \rho^{c,B}_{k_c^B }\, \in\, S_{|B|}^* \big\}_{c,B} \\
  \big\{ \rho^{c,B}_{1}  \cdots  \rho^{c,B}_{k_c^B } = \nu_{c|B} \big\}_{c,B} 
} } \; (-1)^{\sum_{c,B}k_c^B} N^{-nD - \sum_{c,B}\sum_{i_c^B=1}^{k_c^B } \lVert\rho^{c,B}_{i_c^B}\rVert} \;.
\end{align}
where we have exchanged the sums and the products. Note that if $k_c^B=0$ for some $c$ and $B$, then there are no permutations $\rho^{c,B}_{i_c^B}$ and the rightmost sum becomes $\delta_{\nu_{c|B} ; \id}$.

 The permutations $\rho^{c,B}_{i_c^B}$ can be trivially lifted to permutations on $\{1,\dots,n\}$ by supplementing them with the identity on the complement of $B$.
We denote the set of all the (lifted)  permutations $\rho$ by:
\[ 
 \brho = \bigg\{  \rho_{i^B_c}^{c,B}  \bigg| \, 1\le i^B_c \le k_c^B ,\; B \in \pi , \; c\in\{ 1,\dots, D\}  \bigg\}  \; ,
\]
and $\Pi(\brho)$ the partition induced by the transitivity classes of the group generated by all the permutations in $\brho$. 
As $\rho_{i^B_c}^{c,B} $ acts non-trivially only on the block $B\in \pi$, it follows that all the permutations in $\brho$ stabilize the partition $\pi$, hence $\Pi(\brho) \le \pi$. 

Now comes the subtle point. We would like to rewrite $W_\pi^{(N)}[\bsig ,\btau ]$ via a moment-cumulant formula such as \eqref{eq:mom-cum}, that is as a sum over $\pi' \le \pi$ of "cumulants". The obvious idea to reorganize the sum by the partition $ \Pi(\brho)$ of a summand which in turn sums all the $\brho$s with the same $\Pi(\brho)$ fails due to the global factor $\chi\big( \Pi(\bsig,\btau)\le \pi \big)$.  The second idea works: we reorganize the sum by the partition $\Pi(\bsig,\btau) \vee \Pi(\brho) \ge \Pi(\bsig,\btau) $, that is we note that:
\[
W_\pi^{(N)}[\bsig ,\btau ] = \sum_{\Pi(\bsig, \btau) \le \pi' \le \pi}\; W_{C,\pi'}^{(N)}[\bsig ,\btau ] \;,
\]
with the cumulant:
\be
W_{C,\pi'}^{(N)}[\bsig ,\btau ]  
=\sum_{ \{k_c^B \}_{c,B} \ge 0} \;
\sum_{ \substack{ \big\{ \rho^{c,B}_{1} ,\ldots,  \rho^{c,B}_{k_c^B }\, \in\, S_{|B|}^* \big\}_{c,B} \\
  \big\{ \rho^{c,B}_{1}  \cdots  \rho^{c,B}_{k_c^B } = \nu_{c|B} \big\}_{c,B} \\
 \Pi(\bsig,\btau) \vee \Pi(\brho) = \pi'} } \; (-1)^{\sum_{c,B}k_c^B} N^{-nD - \sum_{c,B}\sum_{i_c^B=1}^{k_c^B } \lVert\rho^{c,B}_{i_c^B}\rVert} 
 \; .
\ee
This expression is inverted using \eqref{eq:mom-cum} to yield:
\[
W_{C,\pi}^{(N)}[\bsig ,\btau ]=
\sum_{\Pi(\bsig, \btau) \le \pi'\le \pi}\lambda_{\pi',\pi}
W_{\pi'}^{(N)}[\bsig ,\btau ] \; .\]
Choosing $\pi=1_n$, we recover the right hand side of \eqref{eq:WC-from-Wpi}, thus $W_C^{(N)}[\bsig, \btau ]=W_{C,1_n}^{(N)}[\bsig ,\btau ] $, i.e.:
\begin{align}
\label{eq:explicit-expr-WC}
 W_C^{(N)}[\bsig, \btau ] &= \sum_{\{k_{c}\}_{c}\ge 0} \; 
\sum_{ \substack{ \{ \rho^{c}_1,\ldots , \rho^{c}_{k_{c}}\,\in\, S_{n}^* , \, \rho^{c}_1\cdots\; \rho^{c}_{k_{c}}=\;\nu_{c}  \}_c  \\
 \Pi(\bsig,\btau) \vee \Pi(\brho) = 1_n } }
(-1)^{\sum_{c}k_{c}} N^{-nD - \sum_{c}\sum_{i_c=1}^{k_{c}}\lVert\rho^{c}_{i_c} \rVert} \;,
\crcr &  = \sum_{k,l\ge 0}   (-1)^{k} N^{-nD - l}    \sum_{\substack{{\{k_{c}\}_{c}\ge 0}\\{\sum_c k_c = k}}} \; 
\sum_{ \substack{ \{ \rho^{c}_1,\ldots , \rho^{c}_{k_{c}}\,\in\, S_{n}^* , \, \rho^{c}_1\cdots\; \rho^{c}_{k_{c}}=\;\nu_{c}  \}_c  \\
 \Pi(\bsig,\btau) \vee \Pi(\brho) = 1_n, \; \sum_{c=1}^D\sum_{i_c=1}^{k_c} \lVert\rho^c_{i_c}\rVert =l } }1
\; ,
\end{align}
and we recognize the coefficient of $N^{-nD - l}$ in this expansion to be the alternating sum defining $m_C(\bsig, \btau; l, k)$.

\end{proof}

One drawback of Eq.~\eqref{eq:explicit-expr-WC} is that analytic bounds are difficult to obtain because the sum is signed. On the other hand, it renders obvious the invariance by relabeling of $\{1,\ldots,n\}$. 
\begin{Cor}
\label{cor:WC-is-W}
If $\{\bsig,\btau\}$ act transitively on $\{1,\ldots,n\}$, that is $\lvert \Pi(\bsig,\btau)\rvert=1$, then:
\[
W_C^{(N)}[\bsig, \btau ] \Big\lvert_{ _{\lvert \Pi(\bsig,\btau)\rvert=1}} = \prod_{c=1}^DW^{(N)}[\sigma_c \tau_c^{-1} ]  = \prod_{c=1}^DN^{-n - \lVert\sigma_c\tau_c^{-1}\rVert} \; \M(\sigma_c\tau_c^{-1})(1+O(N^{-2})) \; . 
\]
If moreover $\bsig=\btau$, then  
$ W_C^{(N)}[\bsig, \bsig ] \Big\lvert_{ _{\lvert \Pi(\bsig)\rvert=1}}=  N^{-nD} \big( 1+O(N^{-2} ) \big)  $.
 
\end{Cor}
\begin{proof} If $\Pi(\bsig, \btau)=1_n$, the condition $\Pi(\bsig,\btau) \vee \Pi(\brho) = 1_n$ in \eqref{eq:explicit-expr-WC} is lifted, and we recover the product of \eqref{expansion-wg-permutations} for each color.

\end{proof} 

\begin{Rk}
\label{rk:n-and-mC}
Let $D=1$ and fix $\sigma,\tau\in S_n$. We observe that $\Pi(\sigma, \tau) \ge \Pi(\sigma\tau^{-1})$ 
and  $\Pi(\hat \rho) \ge \Pi(\sigma \tau^{-1})$ for any $\hat \rho$ such that  $\sigma\tau^{-1} = \rho_1\cdots \rho_k$.
This is because the group generated by $\sigma \tau^{-1}$ is a subgroup of both the group generated by $(\sigma,\tau)$ and the one generated by $(\rho_1, \dots , \rho_k) $.

We have that $ \Pi(\sigma,\tau) \vee \Pi(\hat \rho) = \Pi( \hat \rho)$ for any $\hat \rho$ such that $\sigma\tau^{-1} = \rho_1\cdots \rho_k$ if and only if $\Pi(\sigma, \tau) = \Pi(\sigma\tau^{-1})$\footnote{The condition $\Pi(\sigma, \tau) = \Pi(\sigma\tau^{-1})$ means that the bipartite map $(\sigma, \tau^{-1})$ has a single face per connected component: the faces of this map are exactly the cycles of $\sigma\tau^{-1}$ hence correspond to the blocks of $\Pi(\sigma\tau^{-1})$, while the connected components correspond to the blocks of $ \Pi(\sigma, \tau)$.}. This comes about as follows:
\begin{itemize}
    \item if $ \Pi(\sigma,\tau) = \Pi(\sigma\tau^{-1}) $ then,
    taking into account that $\Pi(\hat \rho) \ge \Pi(\sigma \tau^{-1})$, we conclude that $\Pi(\hat \rho) \ge  \Pi(\sigma,\tau) $. therefore $ \Pi(\sigma,\tau) \vee \Pi(\hat \rho) = \Pi( \hat \rho)$.
 \item conversely, assume that $\Pi(\sigma, \tau) > \Pi(\sigma\tau^{-1})$. We will exhibit three permutation 
 $\rho_1,\rho_2,\rho_3$ such that $\rho_1 \rho_2 \rho_3 = \sigma\tau^{-1}$ and such that 
  $ \Pi(\sigma,\tau) \vee \Pi(\hat \rho) > \Pi( \hat \rho)$.
 
As $\Pi(\sigma, \tau) > \Pi(\sigma\tau^{-1})$ there exists a block of $\Pi(\sigma, \tau)$ containing at least two blocks $B_1$ and $B_2$ of $\Pi(\sigma\tau^{-1})$. We chose $\rho_1,\rho_2,\rho_3$
as follows:
\begin{itemize}
 \item $\rho_{1|B_1} = (\sigma\tau^{-1})_{|B_1}$ and $\rho_1$ is the identity on the complement of $B_1$, 
 \item $\rho_{2|B_2} = (\sigma\tau^{-1})_{|B_2}$ and $\rho_2$ is the identity on the complement of $B_2$ 
 \item  $\rho_3$ the identity on $B_1$ and $B_2$ and coinciding with $\sigma\tau^{-1}$
 on their complement:
 \[ 
 \rho_{3|\{1,\dots n\} \setminus B_1\setminus B_2} = (\sigma\tau^{-1})_{|\{1,\dots n\} \setminus B_1\setminus B_2} \;.\] 
\end{itemize}
Obviously $\sigma\tau^{-1} = \rho_1\rho_2\rho_3$. At the same time $ \Pi(\sigma,\tau) \vee \Pi(\hat \rho) >  \Pi( \hat \rho)$ because the blocks $B_1$ and $B_2$ which are distinct blocks in $\Pi(\hat \rho)$ are collapsed into one block of $\Pi(\sigma,\tau)$.
 \end{itemize}
 
Now, at $D=1$, for any $\sigma$ and $\tau$ we have:
\be\label{eq:mmcc}
m_C(\sigma,\tau; l, k )  = 
\sum_{ \substack{ \rho_1,\ldots , \rho_{k}\,\in\, S_{n}^* , \, \rho_1\ldots\; \rho_{k}=\;\sigma\tau^{-1}  \\
   \sum_{i=1}^{k}\lVert\rho_{i} \rVert =l \;, \;\; 
   \Pi(\sigma,\tau) \vee \Pi(\hat \rho) = 1_n } 
 } 1 \; ,
\ee
and, using on the one hand the remark above and on the other noting that $\Pi(\sigma\tau^{-1})\le \Pi(\hat \rho)$ we get:
\be 
\label{eq:gamma-and-pC-D1}
\begin{split}
& m_C(\sigma,\tau; l, k) \Big{|}_{
 \Pi(\sigma,\tau)   = \Pi(\sigma\tau^{-1})}   =  m_C(\sigma\tau^{-1}, \mathrm{id} ; l, k) = \gamma_l(\sigma\tau^{-1} ; k )   \;,
\crcr
& (-1)^{l}  p_C[ \sigma, \tau; l ] \Big{|}_{ \Pi(\sigma,\tau) =  \Pi (\sigma\tau^{-1})} = 
 (-1)^{l}  p_C[ \sigma\tau^{-1}, \mathrm{id}; l ]  
= \gamma_l(\sigma\tau^{-1} )      \;,
\end{split}
\ee
as in both cases the condition $ \Pi(\sigma,\tau) \vee \Pi(\hat \rho) = 1_n$ in  \eqref{eq:mmcc} reduces to $ \Pi(\hat \rho) = 1_n $.

\end{Rk}

\

There is a non-signed version in terms of a generalization of monotone double Hurwitz numbers, which we describe now (see also Sec.~\ref{sub:Hurwitz}). 
\begin{Prop} 
\label{def:PC}
The number $p _C[\bsig,\btau ;l]$ in Thm.~\ref{thm:1Nexpansion-Weingarten-Cumulants} is also the number of $D$ ordered sequences of \emph{transpositions} $ (\mu^1_1,\ldots, \mu^1_{l_1}), \ldots,  (\mu^D_1,\ldots, \mu^D_{l_D}) $ such that:
\begin{itemize}
 \item for every $c\in\{1,\dots, D\}$, the $l_c$ transpositions $ \mu^c_1,\ldots, \mu^c_{l_c} $ have weakly monotone maxima (Def.~\ref{def:weakly-monotone}) and satisfy $\sigma_c = \mu^c_1 \ldots \mu^c_{l_c}\tau_c$, 
\item $l_1+\ldots + l_D=l$,
\item the group generated by all the transpositions $\mu$ and all $\sigma_c$ and $\tau_c$ is transitive on  $\{ 1,\ldots , n\}$.
\end{itemize}
In particular $p _C[\bsig,\btau ;l]$ is a non-negative integer. 
\end{Prop}
\begin{proof}
Note that 
\begin{align}
 W_\pi^{(N)}[\bsig ,\btau ] = \chi\big( \Pi(\bsig,\btau)\le \pi \big) \sum_{ \{l_c^B \}_{c,B} \ge 0} \;
\sum_{ \big\{ \mu^{c,B}_{1} ,\ldots,  \mu^{c,B}_{l_c^B }\, \in\, P(\nu_{c|B}, l_c^B) \big\}_{c,B} 
}  \; (-N^{-1})^{\sum_{c,B}l_c^B}  \;,
\end{align}
where $P(\nu, l) $ denotes the set of transpositions with weakly monotone maxima which factorize $\nu$ (Def.~\ref{def:weakly-monotone}).
From this point, the proof is  {\it mutatis mutandis} the same as that of  Thm.~\ref{thm:1Nexpansion-Weingarten-Cumulants}.

\end{proof}

\newpage

\section{Asymptotics of the cumulant Weingarten functions.} 
\label{sub:asympt-cum-weing}

In many applications, such as random tensor models, one is interested in the first place in computing the large-$N$ contribution to the logarithm 
${\cal C}_{D,N}(t, A, B)$ of the tensor HCIZ integral \eqref{eq:HCIZTens1}. For some given $\bsig, \btau\in \bS_n$, one thus needs to identify the smallest integer $l$ such that $p_C[\bsig, \btau; l]$  does not vanish and, if possible, to obtain an explicit expression for the corresponding $p_C$. We provide this  in Thm.~\ref{thm:asympt-cum-weing}. A general combinatorial formula, \eqref{formula-gen-l}, is furthermore derived for $p_C[\bsig, \btau; l]$ for any $l$. To our knowledge this expression for the sub-leading contributions to the cumulant Weingarten functions is new also in $D=1$.

\subsection{Main result} 

The large $N$ behavior of the cumulant Weingarten functions
is captured by the following theorem.

\begin{Th} 
\label{thm:asympt-cum-weing} 

For any $l$, the coefficient $p_C [\bsig, \btau  ;  l ]$ is given by:
\be
\label{eq:formula}
(-1)^l \;p_C [\bsig, \btau  ;  l ] =
\sum_{\substack{{\pi_1\ge \Pi(\nu_1) ,\ \ldots\ ,\  \pi_D\ge \Pi(\nu_D)}\\{\Pi(\bsig, \btau)\vee\pi_1\vee\ldots\vee\pi_D = 1_n } } } \; \sum_{\substack{  l_1,\dots, l_D \ge 0  \\ \sum l_c = l } } \; 
\prod_{c=1}^D
\left[
 \sum_{ \substack{ \{ l_{B_c} \}_{B_c \in \pi_c }\ge 0     \\ \sum_{B_c} l_{B_c}  = l_c } }  \;   
  \prod_{B_c\in \pi_c} \gamma_{l_{B_c} }(\nu_{c|B_c})  \right]  \; ,
\ee 
with $\nu_c = \sigma_c\tau_c^{-1}$ and $\gamma_l(\nu)$ defined in Sec.~\ref{sec:const}.
The smallest value of $l$ such that $p_C [\bsig, \btau  ;  l ]$  does not vanish is:
\be
\label{eq:l-minimal}
\ell(\bsig, \btau)=\sum_{c=1}^D \lVert\sigma_c\tau_c^{-1}\rVert  + 2  \big(\lvert \Pi(\bsig,\btau)\rvert-1 \big) \; .
\ee
In order to simplify the notation we sometimes denote $\ell \equiv \ell(\bsig,\btau)$.
The cumulant Weingarten functions thus have the asymptotic expression:
\begin{equation} 
\label{eq:WC-asym}
W_C^{(N)}[\bsig , \btau ]= N^{-nD-\ell}  \, (-1)^{\ell \;}
p_C [\bsig, \btau ; \ell ]  \, (1+O(N^{-2})) \; ,
\end{equation}
where the leading order  coefficient is:
\be
\label{eq:comb-expr-leading-cum-Weing}
(-1)^{\ell  } \;  \pC = 
\sum_{\substack{{\pi_1\ge \Pi(\nu_1) ,\ \ldots\ ,\  \pi_D\ge \Pi(\nu_D)}\\{\Pi(\bsig, \btau)\vee\pi_1\vee\ldots\vee\pi_D = 1_n }\\{\sum_c (|\Pi(\nu_c)| - \lvert\pi_c\rvert)=\lvert\Pi(\bsig, \btau)\rvert - 1}}} \ 
 \prod_{c=1}^D \prod_{B_c\in \pi_c} \gamma(\nu_{c|B_c}) \;,
\ee
Note that $\nu_{c|B_c}$ is well-defined as $\pi_c \ge \Pi(\nu_c)$. In detail:
\be
\gamma(\nu_{c|B_c}) = \M(\nu_{c|B_c}) \; \frac{( 3 \lvert B_c \rvert - \lVert\nu_{c|B_c} \rVert - 3)!}{(2\lvert B_c \rvert )!}   \,
\prod_{p\ge 1} \big[2p (2p-1)\big]^{d_p(\nu_{c|B_c})}   \; , 
\ee 
with the non-crossing M\"oebius function $\M$ defined in \eqref{eq:Moebius-on-NC}. 

In general for $l>\ell(\bsig,\btau)$, \eqref{eq:formula} can be rewritten as:
\be
\label{formula-gen-l}
\begin{split}
&(-1)^l \;p_C [\bsig, \btau  ;  l ]  = \crcr
& = \sum_{L=0}^{ \tfrac{ l-\ell  } {2} } \sum_{\substack{{\pi_1\ge \Pi(\nu_1) ,\ \ldots\ ,\  \pi_D\ge \Pi(\nu_D)}\\{\Pi(\bsig, \btau)\vee\pi_1\vee\ldots\vee\pi_D = 1_n } \\{\sum_c (|\Pi(\nu_c)| - \lvert\pi_c\rvert)=\lvert\Pi(\bsig, \btau)\rvert + L  - 1}} } \;
 \sum_{\substack{  g_1,\dots, g_D \ge 0  \\ \sum_c g_c = \tfrac{l-\ell }{2} - L } } \; 
\prod_{c=1}^D
\left[
 \sum_{ \substack{ \{ g_{B_c} \}_{B_c \in \pi_c }\ge 0     \\ \sum_{B_c} g_{B_c}  = g_c } }  \;   
  \prod_{B_c\in \pi_c} \gamma_{l(g_{B_c}) }(\nu_{c|B_c})  \right]  \; ,
  \end{split}
\ee 
where $l(g_{B_c}) = \lvert B_c \rvert + \Pi(\nu_{c|B_c}) + 2g_{B_c} - 2$. 

\end{Th}

\begin{proof}
 See Sec.~\ref{sub:proof-of-the-theorem}.
 
\end{proof}

The sum over partitions appears rather complicated, however it has a simple graphical interpretation in terms of sums of trees. This graphical interpretation was  developed in \cite{Zub-ZJ} in $D=1$ and for $l=\ell(\sigma, \tau)$ and is generalized in this paper to larger $l$ and larger $D$ in Sec.~\ref{sec:nodal} (more precisely Sec.~\ref{sec:nodal-pC}).

\ 

Corollary~\ref{cor:WC-is-W} implies that if $\lvert\Pi(\bsig, \btau)\rvert=1$ then
$(-1)^{\ell}\pC = \prod_{c=1}^D\M(\nu_c)$. We have chosen to factor the M\"oebius functions to render this explicit. This can be obtained directly from \eqref{eq:comb-expr-leading-cum-Weing}: as $\pi_c \ge \Pi(\nu_c)$, we have $\lvert\Pi(\nu_c)\rvert \ge \lvert\pi_c\rvert$ and, from the condition in the sum, $\lvert\Pi(\nu_c)\rvert = \lvert\pi_c\rvert$ so that only $\pi_c = \Pi(\nu_c)$ contributes, and $\nu_{c|B_c}$ is the restriction of $\nu_c$ to one of its cycles:
\[
\frac{( 3 \lvert B_c \rvert - \lVert\nu_{c|B_c} \rVert - 3)!}{(2\lvert B_c \rvert )!}   \,
\prod_{p\ge 1} \big[2p (2p-1)\big]^{d_p(\nu_{c|B_c})}  =1 \;.
\]

The leading contribution at large $N$ to the cumulant is:
\begin{align}
\label{eq:WC-asym2}
 C\bigg( [ \Tr (AUBU^* )]^n \bigg)  =  N^{-nD} \sum_{  \bsig,\btau \, \in \bS_{n} }
& \Tr_{\bsig}(A)  \Tr_{\btau^{-1}}(B) \ N^{ - \sum_{c=1}^D \lVert\sigma_c\tau_c^{-1}\rVert- 2 (\lvert\Pi(\bsig, \btau)\rvert-1 ) }  \\
&  \qquad  (-1)^{\ell} \pC \ (1+O(N^{-2}) ) \; . \nonumber
\end{align}

As a function of the scaling behavior of the trace invariants $\Tr_{\bsig}(A)$ and $\Tr_{\btau^{-1}}(B) $ with $N$, the sum in \eqref{eq:WC-asym2} is dominated by a subset of the terms. For instance if $\Tr_{\bsig}(A) \sim \Tr_{\btau^{-1}}(B) \sim O(1)$ in the limit of large $N$, then the term with $\bsig = \btau$, $\| \Pi(\bsig)\| = 1$ will dominate. If $\Tr_{\bsig}(A) \sim 1 $ but $\Tr_{\btau^{-1}}(B) \sim N^{\sum_c \#(\tau_c)} $, more terms dominate at large $N$.
A detailed study of the possible behaviors of the cumulant, relevant for different applications to physics will be conducted in future work.

\subsection{Proof of Theorem \ref{thm:asympt-cum-weing}} 
\label{sub:proof-of-the-theorem}

The proof of Theorem  \ref{thm:asympt-cum-weing} is divided into four parts:
\begin{itemize}
 \item[-]{\it Derivation of \eqref{eq:formula}.} Although lengthy, this part is straightforward: we compute the sum over $k$ in Theorem \ref{thm:1Nexpansion-Weingarten-Cumulants}. We obtain \eqref{eq:formula} where the right hand side is an alternating sum (with constraints) over constellations. 
 \item[-] {\it Reinterpretation of  $\Pi(\bsig, \btau)\vee\pi_1\vee\ldots\vee\pi_D = 1_n $}. We show that the condition:
\[
\Pi(\bsig, \btau)\vee\pi_1\vee\ldots\vee\pi_D = 1_n \;,
\]
 which constrains the sum over $\{\pi_c\}_c$ in \eqref{eq:formula} is equivalent to requiring that a certain abstract graph, aptly denoted $ G \big[\Pi(\bsig, \btau) , \{\pi_c\}_c ; \Pi(\nu_c) \big] $, is connected.
  \item[-]{\it The lower bound \eqref{eq:l-minimal} on $l$.} We show that the leading order at large $N$ (minimal $l$) in \eqref{eq:formula} fulfills  two conditions:
   \begin{itemize}
       \item the abstract graph $ G \big[\Pi(\bsig, \btau) , \{\pi_c\}_c ; \Pi(\nu_c) \big] $ has minimal number of edges. As it is connected, this means it is a tree.
    \item the constellations are planar.
   \end{itemize}
We will show in Section~\ref{sec:nodal} that these two conditions translate in fact the planarity of a certain nodal surface.

 \item[-]{\it Proof of \eqref{formula-gen-l}.} We reorganize the terms in  \eqref{eq:formula} by the number of excess edges of the abstract graph $ G \big[\Pi(\bsig, \btau) , \{\pi_c\}_c ; \Pi(\nu_c) \big] $ (\emph{i.e.} the number of independent cycles, or loop edges in the physics literature) and by the genera of the constellations to get 
  \eqref{formula-gen-l}.
\end{itemize}

\paragraph{Derivation of \eqref{eq:formula}.}
Our starting point is Theorem \ref{thm:1Nexpansion-Weingarten-Cumulants}, which states that:
\[
 W_C^{(N)}[\bsig, \btau ]= N^{-nD} \sum_{l\ge 0} (-1)^l \; p_C[\bsig,\btau;l] \;N^{-l}
  \; ,\quad (-1)^l \; p_C[\bsig,\btau,l]  = \sum_{k\ge 0} (-1)^k m_C(\bsig,\btau; l, k )  \;,
\] 
where, denoting $\nu_c = \sigma_c\tau_c^{-1}$, we have:
\[
m_C(\bsig,\btau; l, k )  = \sum_{\{k_{c}\}_{c}\ge 0} \; 
\sum_{ \substack{ \{ \rho^{c}_1,\ldots , \rho^{c}_{k_{c}}\,\in\, S_{n}^* , \, \rho^{c}_1\cdots\; \rho^{c}_{k_{c}}=\;\nu_{c}  \}_c  \\
   \sum_{c}\sum_{i_c=1}^{k_{c}}\lVert\rho^{c}_{i_c} \rVert =l  , \; \sum_{c} k_c = k \\
    \Pi(\bsig,\btau) \vee \Pi(\brho) = 1_n } 
 } 1 \; ,
\]
and $\brho$ denotes the $D$-uple of constellations 
$(\hat \rho^1, \dots ,\hat \rho^D)$ where, for $ c\in\{1,\dots, D\}$, the constellation $\hat \rho^c$ is $\hat \rho^c= (\rho^c_1, \dots, \rho^c_{k_c})$. 
We aim to derive the asymptotic behaviour of $ W_C^{(N)}[\bsig, \btau ]$ using the results in Sec.~\ref{sec:const}. Let us classify the terms in the above formula by the values 
$l_c = \sum_{i_c=1}^{k_{c}}\lVert\rho^{c}_{i_c} \rVert$ and by the partitions 
$\pi_c = \Pi(\hat \rho^c) \ge \Pi(\nu_c)$: 
\begin{align}
 m_C(\bsig,\btau; l, k )  & =
  \sum_{ \substack{ \pi_1\ge \Pi(\nu_1) ,\dots, \pi_D \ge \Pi(\nu_D) \\ \Pi(\bsig,\btau) \vee \pi_1\dots \vee \pi_D = 1_n } } \; \;
 \sum_{\substack{  l_1, \dots, l_D\ge 0  \\ \sum_c l_c = l} } \;
\sum_{ \substack{ \{ k_c \}_c \ge 0  \\ \sum_c k_c = k} }  \; \prod_{c=1}^D\, M(\pi_c, \nu_c ; l_c, k_c) \;,
\nonumber
\\
M(\pi_c, \nu_c ; l_c, k_c)  & = 
 \sum_{ 
 \substack{ \rho^{c}_1,\ldots , \rho^{c}_{k_{c}}\,\in\, S_{n}^* , \, \rho^{c}_1\cdots\; \rho^{c}_{k_{c}}=\;\nu_{c}   \\
  \Pi( \hat \rho^c) = \pi_c , \;   \sum_{i_c=1}^{k_{c}}\lVert\rho^{c}_{i_c} \rVert =l_c  } } 1\;.
 \label{eq:M}
\end{align}

We wish to compute 
\begin{align}
\quad (-1)^l \; p_C[\bsig,\btau,l]  &= \sum_{k\ge 0} (-1)^k m_C(\bsig,\btau; l, k )\nonumber  \\
& = \sum_{ \substack{ \pi_1\ge \Pi(\nu_1) ,\dots, \pi_D \ge \Pi(\nu_D) \\ \Pi(\bsig,\btau) \vee \pi_1\dots \vee \pi_D = 1_n } } \; \;
 \sum_{\substack{  l_1, \dots, l_D\ge 0  \\ \sum_c l_c = l} } \;
  \prod_{c=1}^D\, \sum_{  \{ k_c \}_c \ge 0  } (-1)^{k_c}M(\pi_c, \nu_c ; l_c, k_c)\;.
  \label{sum-with-S}
\end{align}

\paragraph{\it The sum $\sum_{k\ge 0} (-1)^k  M(\pi, \nu ; l, k)$.}  Let us focus on the rightmost sum of \eqref{sum-with-S}. We define: 
\begin{align*}
& S[ \pi, \nu ; l] =\sum_{k \ge 0} (-1)^{k}  M ( \pi, \nu ;l,k)  \;, \qquad
  M( \pi , \nu; l, k)  = \sum_{\substack{ \rho_1,\dots, \rho_k \in S_n^*, \rho_1\cdots \rho_k = \nu \\ 
    \sum_i\lVert\rho_i\rVert = l \;,  \Pi(\hat \rho) = \pi 
  }} 1 \; ,  \crcr
 & \qquad \tilde     M( \pi , \nu; l, k)  = \sum_{\substack{ \rho_1,\dots, \rho_k \in S_n, \rho_1\cdots \rho_k = \nu \\ 
    \sum_i\lVert\rho_i\rVert = l \;,  \Pi(\hat \rho) = \pi 
  }} 1 \; , 
\end{align*}
where we emphasize that, contrary to $M$, the sum defining $\tilde M$ includes the case when some of the permutations $\rho_i$ are the identity. $M$ and $\tilde M$ are related by:
\[ 
  \tilde M(\pi,\nu ; l , q) = \sum_{k=0}^q \binom{q}{k}  \; M( \pi , \nu; l, k)  \;, \quad
 M( \pi , \nu; l, k) = \sum_{q = 0}^k \binom{k}{q} (-1)^{k-q} \; \tilde M(\pi , \nu; l , q) \; .
\]
$\tilde M(\pi,\nu ; l , q)$ is a sum over $q$ permutations $\rho_i$. The first equation follows by noting that if exactly $q-k$ out of these permutations are the identity then 
$\tilde M(\pi,\nu ; l , q)$ reduces to $M(\pi,\nu ; l , k)$; the second equation is obtained by inverting the first one.

The point is that $\tilde M$ factors over the blocks of $\pi$. As $\nu$ (respectively $\rho_i$) stabilizes any block $B\in \pi$, it can be decomposed as the product of $|\pi|$ permutations $\nu_{|B}$ 
(respectively $\rho_{i|B}$),  
$\nu =\prod_B \nu_{|B}$ (respectively $\rho_i = \prod_B \rho_{i|B} $),
where we lift trivially $\nu_{|B}$ (respectively $\rho_{i|B}$) to the whole set $\{1,\dots, n\}$. The number of transpositions of $\rho_i$ is distributed among the blocks of $\pi$,
$\lVert\rho_i\rVert = \sum_B \lVert \rho_{i|B}\rVert$ and
we get:
\[
   \sum_{\substack{ \rho_1,\dots, \rho_q \in S_n, \rho_1\cdots \rho_q = \nu \\ 
    \sum_i\lVert\rho_i\rVert = l \;, \Pi(\hat \rho) = \pi  }} 1
    =\sum_{ \substack{ \{ l_B  \}_{B } \ge 0    \\ \sum_{B} l_B = l } }  \; \; \prod_{B\in \pi} \sum_{\substack{ \rho_{1|B}, \dots, \rho_{q|B}  \in S_{|B|} \\
    \rho_{1|B} \cdots \rho_{q|B}  = \nu_{|B } \\ 
    \sum_i\lVert\rho_{i|B} \rVert = l_B \;, \Pi(\hat \rho_{|B}) = 1_{|B|}   }} 1 
     = \sum_{ \substack{ \{ l_B  \}_{B } \ge 0    \\ \sum_{B} l_B = l } }  \; \;    \prod_{B\in \pi} 
     \tilde \gamma_{l_B}(\nu_{|B}; q) \;,
\]
where we recognized the number of connected $q$-constellations with fixed $l_B$, 
$    \tilde \gamma_{l_B}(\nu_{|B};  q)  $ of Sec.~\ref{sec:const}. We emphasize that constellations are not necessarily proper (i.e.~the sums run over $S_{|B|}$ not 
$S_{|B|}^*$). The number of arbitrary (\emph{i.e.}~not necessarily proper) constellations is written  in terms of the number of proper ones as:
\[
  \tilde \gamma_{l_B}(\nu_{|B};  q)  = \sum_{k_B =0}^q \binom{q}{k_B} \gamma_{l_B}(\nu_{|B};  k_B)\; ,
\] 
and substituting, we find:
\begin{equation*}
\begin{split}
& S[ \pi, \nu ; l]   =
\sum_{k \ge 0} \sum_{q=0}^k \binom{k}{q} (-1)^q  \tilde M(\pi, \nu ; l, q) 
  =  \sum_{ \substack{ \{ l_B \}_B \ge 0   \\ \sum_{B} l_B = l } }  \;  \sum_{k \ge  0} \sum_{q=0}^k \binom{k}{q} (-1)^q 
 \prod_{B\in \pi}       \tilde \gamma_{l_B}(\nu_{|B};  q)   \crcr
& \qquad  = \sum_{ \substack{ \{ l_B  \}_B \ge 0   \\ \sum_{B} l_B = l } }  \;  
\sum_{  \{  k_B \}_B \ge  0   }  \; 
 \left(  \prod_{B\in \pi}   \gamma_{l_B}(\nu_{|B};  k_B)  \right)
 \left[ \sum_{k \ge 0 }
\sum_{q=0}^k \binom{k}{q} (-1)^q 
 \prod_{B\in \pi}  \binom{q}{k_B}  \right] \; ,
\end{split}
\end{equation*}
with the convention that ill-defined binomial coefficients (e.g.~$k_B >q$) are zero.
At fixed $k_B$, the sum over $q$ and $k$ can be computed as it is the coefficient of the monomial $ \prod_B x_B^{k_B}$ in the generating function:
\[
 \sum_{k \ge 0 } \sum_{q=0}^k \binom{k}{q} (-1)^q  \prod_{B} 
\left[ \sum_{j_B=0}^q\binom{q}{j_B} x_B^{j_B} \right]
= \sum_{k \ge 0 }  \big[ 1 - \prod_B (1+x_B)\big]^k  = \sum_{ \{ k_B  \}_B  \ge 0}  \prod_B (-1)^{k_B} x_B^{k_B} \; ,
\]
which ultimately leads to:
\[ 
 S[ \pi, \nu ; l]   = \sum_{ \substack{ \{ l_B  \}_B  \ge 0  \\ \sum_{B} l_B = l } }  \;  
\sum_{  \{  k_B \}_B  \ge  0  }  \prod_{B\in \pi}    (-1)^{k_B} \gamma_{l_B}(\nu_{|B} ; k_B)  
     = \sum_{ \substack{ \{ l_B  \}_B  \ge 0  \\ \sum_{B} l_B = l } }  \;   \prod_{B\in \pi} 
     \gamma_{l_B}(\nu_{|B}) \;.
\]
Inserting this expression in \eqref{sum-with-S} achieves the proof of \eqref{eq:formula}:
\[ 
(-1)^l \;p_C [\bsig, \btau  ;  l ] =
\sum_{\substack{{\pi_1\ge \Pi(\nu_1) ,\ \ldots\ ,\  \pi_D\ge \Pi(\nu_D)}\\{\Pi(\bsig, \btau)\vee\pi_1\vee\ldots\vee\pi_D = 1_n } } } \; \sum_{\substack{  l_1,\dots, l_D \ge 0  \\ \sum l_c = l } } \; 
\prod_{c=1}^D
\left[
 \sum_{ \substack{ \{ l_{B_c} \}_{B_c \in \pi_c }\ge 0     \\ \sum_{B_c} l_{B_c}  = l_c } }  \;   
  \prod_{B_c\in \pi_c} \gamma_{l_{B_c}}(\nu_{c|B_c} )  \right]  \; .
\] 

This formula can be analyzed further.
\paragraph{The graph $G[\Pi,\{ \pi_c\}_c ; \{\Pi_c\}_c]$.} 
Consider $2D+1$ partitions $\Pi,\{\pi_c\}_c,\{\Pi_c\}_c $ on $\{1,\dots, n\}$ such that:
\[
\forall c,\quad  \Pi\ge \Pi_c \quad \textrm{and}\quad \pi_c \ge \Pi_c \;.
\]
The relation between these partitions can be encoded in a convenient graphical representation.
Let $G[\Pi,\{ \pi_c\}_c ; \{\Pi_c\}_c]$ be the abstract bipartite graph consisting in:
\begin{itemize}
 \item white vertices associated to the blocks $B$ of the partition $\Pi$,
 \item $c$--colored vertices associated to the blocks $B_c$ of the partitions $\pi_c$
 \item $c$--colored edges associated to the blocks $b_c$ of $\Pi_c$ linking a white and a $c$-colored vertex. The block $b_c$ is  at the same time:
 \begin{itemize}
     \item  contained in a block of $\Pi $, which we denote by $B(b_c)$, as $\Pi\ge \Pi_c$,
     \item  contained in a block of $\pi_c$, which we denote by  $B_c(b_c)$, as $\pi_c \ge \Pi_c$.
 \end{itemize}
The edge corresponding to $b_c$ links $B(b_c)$ to $B_c(b_c)$. 
 
\end{itemize}
 
 \begin{figure}[!h]
\centering
\includegraphics[scale=1]{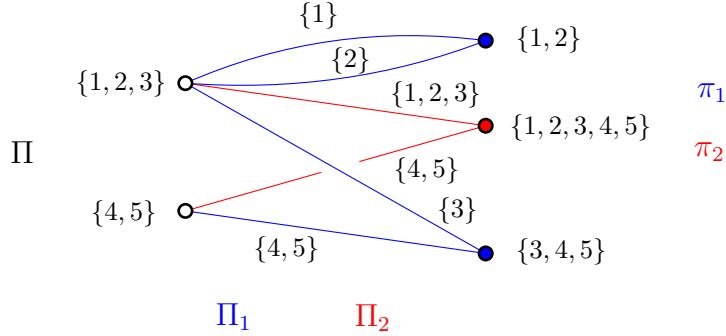} 
\caption{The graph $G[\Pi , \{ \pi_c\}_c ; \{\Pi_c\}_c ]$ for $\Pi=\big\{ \{1,2,3\}, \{4,5\} \big\}$, $\pi_1= \big\{ \{1,2\},\{3,4,5\} \big\} $, 
$\pi_2= \big\{ \{1,2,3,4,5\} \big\} $ and  $\Pi_1= \big\{ \{1\},\{2\},\{3\},\{4,5\} \big\} $, 
$\Pi_2= \big\{ \{1, 2, 3\} , \{4,5\} \big\}$. Here $D=2$ (2 colors), $n=5$.
}
\label{fig:ex-incidence-graph}
\end{figure}

\begin{Lem}\label{lem:lemcon}
The graph $G[\Pi, \{ \pi_c\}_c; \{ \Pi_c\}_c ]$ has $ | \Pi\vee \pi_1 \dots \vee \pi_D|$ connected 
components. 
\end{Lem}

\proof
Two blocks $B\in \Pi$ and $B_c \in \pi_c$ are connected by an edge in $G[\Pi, \{ \pi_c\}_c; \{ \Pi_c\}_c ]$ if and only if there exists a $b_c\in \Pi_c$ such that $b_c\subset B, B_c$ hence both $B$ and $B_c$ belong to the block of $  \Pi\vee \pi_1 \dots \vee \pi_D$ which contains $b_c$. The Lemma follows by noting that:
\begin{itemize}
 \item[-] belonging to the same connected component of $G[\Pi, \{ \pi_c\}_c; \{ \Pi_c\}_c ]$,
 \item[-] belonging to the same block of the partition $  \Pi\vee \pi_1 \dots \vee \pi_D$,
\end{itemize}
are both transitive relations between the blocks of $\Pi$ and 
$\{\pi_c\}_c$. 

\qed

\paragraph{Lower bound on $l$.}  In order to find a lower bound on $l$, we first rewrite $\sum_c |\pi_c|$.  Observe that in \eqref{eq:formula} we sum over partitions $\{\pi_c\}_c$ such that:
\[
  \forall c \; \pi_c \ge \Pi(\nu_c) \,,\quad  \textrm{and}\quad 
  \Pi(\bsig, \btau) \vee \pi_1 \dots \vee \pi_D = 1_n \;  .
\]
Since $\Pi(\bsig, \btau) \ge \Pi(\nu_c)$, from Lemma \ref{lem:lemcon} we conclude that the sum runs over partitions $\{\pi_c\}_c$ such that the graph $ G \big[\Pi(\bsig, \btau) , \{\pi_c\}_c ; \Pi(\nu_c) \big] $ is connected. 

The graph $ G \big[\Pi(\bsig, \btau) , \{\pi_c\}_c ; \Pi(\nu_c) \big] $ has $  \sum_{c} |\Pi(\nu_c)|$ edges and $ |\Pi(\bsig, \btau) | +  \sum_c |\pi_c| $ vertices. If it is connected then any tree spanning this graph will have exactly $ |\Pi(\bsig, \btau) | +  \sum_c |\pi_c| -1 $ edges. We denote the number of excess edges of  $ G \big[\Pi(\bsig, \btau) , \{\pi_c\}_c ; \{ \Pi(\nu_c) \}_c \big] $, that is the number of edges in the complement of a spanning tree in the graph, by: 
\be
\label{eq:L}
 L \big[ \Pi(\bsig, \btau) , \{\pi_c\}_c ; \{\Pi(\nu_c)\}_c \big] 
 =   \sum_{c} |\Pi(\nu_c)|    -  |\Pi(\bsig, \btau) | -  \sum_c |\pi_c|  + 1 \ge 0 \;.
\ee

Let us consider a term in \eqref{eq:formula}. At fixed $l_{B_c}$, each block $B_c$ of $\pi_c$ contains a sum over constellations with fixed genus as, from \eqref{eq:Euler-constellations-lengths}:
\be\label{eq:truc3}
   l_{B_c} = 2 |B_c| - \lVert \nu_{c|B_c} \rVert + 2g_{B_c} -2 \;, 
  \qquad l_c = 2n  - \lVert \nu_c \lVert + 2 \sum_{B_c\in \pi_c} g_{B_c} - 2 |\pi_c| \;.
\ee
Summing over $c$ and using $ n = \lVert \nu_c \lVert + |\Pi(\nu_c)|$ we get:
\be\label{eq:truc}
l = \sum_c \lVert\nu_c\rVert  + 2 \Bigl( |\Pi(\bsig,\btau)| -1 \Bigr) + 2 \sum_{c=1}^D  \sum_{B_c\in \pi_c} g_{B_c} 
 + 2 L \big[ \Pi(\bsig, \btau) , \{\pi_c\}_c ; \{ \Pi(\nu_c) \}_c\big] \;.
\ee
Thus  $l\ge \ell(\bsig,\btau)$ leading to Eq.~\eqref{eq:l-minimal}, which proves \eqref{eq:WC-asym}. We see that  the bound $l= \ell(\bsig,\btau)$ is attained if and only if $ g_{B_c} =0$ for all $B_c$, and: 
\[  L \big[ \Pi(\bsig, \btau) , \{\pi_c\}_c ; \{ \Pi(\nu_c) \}_c \big] =0\quad  \Rightarrow \quad \sum_{c} |\Pi(\nu_c)|    -  \sum_c |\pi_c|   =  |\Pi(\bsig, \btau) |  -1 \;, \]
which proves \eqref{eq:comb-expr-leading-cum-Weing}.

\begin{Rk}
The condition $g_{B_c}=0$ for all $B_c$ already suggests that the large-$N$ limit corresponds to some type of planarity. It turns out that the additional condition $ L \big[ \Pi(\bsig, \btau) , \{\pi_c\}_c ; \{ \Pi(\nu_c) \}_c \big] =0 $
is also a ``minimal genus'' condition. In Section~\ref{sec:nodal} we will show that the sum in \eqref{eq:formula} can be reinterpreted as a sum over a class of \emph{nodal} surfaces and at leading order at large-$N$ only nodal surfaces of minimal arithmetic genus contribute.  
\end{Rk}

  \paragraph{Formula for $l>\ell \equiv \ell(\bsig, \btau)$.} More generally for non-minimal $l$, we can organize the sum in \eqref{eq:formula} by the number of excess edges of the graph $ G \big[\Pi(\bsig, \btau) , \{\pi_c\}_c ; \{ \Pi(\nu_c) \}_c \big] $. Writing \eqref{eq:truc} as:
  \be\label{eq:truc2}
\frac{l - \ell}2 =  \sum_{c=1}^D  \sum_{B_c\in \pi_c} g_{B_c} 
 +  L \big[ \Pi(\bsig, \btau) , \{\pi_c\}_c ; \{ \Pi(\nu_c) \}_c\big] \;,
 \ee
we reorganize \eqref{eq:formula} as:
\[ 
(-1)^l \;p_C [\bsig, \btau  ;  l ] =
\sum_{\substack{{\pi_1\ge \Pi(\nu_1) ,\ \ldots\ ,\  \pi_D\ge \Pi(\nu_D)}\\{\Pi(\bsig, \btau)\vee\pi_1\vee\ldots\vee\pi_D = 1_n } } } \; \sum_{\substack{  l_1,\dots, l_D \ge 0  \\ \sum l_c = l } } \; 
\prod_{c=1}^D
\left[
 \sum_{ \substack{ \{ l_{B_c} \}_{B_c \in \pi_c }\ge 0     \\ \sum_{B_c} l_{B_c}  = l_c } }  \;   
  \prod_{B_c\in \pi_c} \gamma_{l_{B_c}}(\nu_{c|B_c} )  \right]  \; ,
\] 
by replacing the sum over the partitions $\{\pi_c\}_c$ with a sum over $\{\pi_c\}_c$ such that $L = \sum_c \big( |\Pi(\nu_c)| - |\pi_c| \big)- \Pi(\bsig,\btau) +1 $ is fixed and a sum over $L$ from its minimal allowed value $0$ to the maximal allowed value $(l-\ell)/2$ fixed by \eqref{eq:truc2}. Then at fixed $L$, 
we use \eqref{eq:truc3} and trade the sums over 
$\{l_c\}_c, \sum_{c}l_c=l$ with sums over the genera $g_{B_c}\ge 0$ constrained by \eqref{eq:truc3} to obey:
\[
\sum_{c=1}^D  \sum_{B_c\in \pi_c} g_{B_c}  = \frac{ l -\ell}{2}  - L \; .
\]
At the end of the day, we get:
\[
\begin{split}
&(-1)^l \;p_C [\bsig, \btau  ;  l ]  = \crcr
& = \sum_{L=0}^{ \tfrac{ l-\ell  } {2} } \sum_{\substack{{\pi_1\ge \Pi(\nu_1) ,\ \ldots\ ,\  \pi_D\ge \Pi(\nu_D)}\\{\Pi(\bsig, \btau)\vee\pi_1\vee\ldots\vee\pi_D = 1_n } \\{\sum_c (|\Pi(\nu_c)| - \lvert\pi_c\rvert)=\lvert\Pi(\bsig, \btau)\rvert + L  - 1}} } \;
 \sum_{\substack{  g_1,\dots, g_D \ge 0  \\ \sum_c g_c = \tfrac{l-\ell }{2} - L } } \; 
\prod_{c=1}^D
\left[
 \sum_{ \substack{ \{ g_{B_c} \}_{B_c \in \pi_c }\ge 0     \\ \sum_{B_c} g_{B_c}  = g_c } }  \;   
  \prod_{B_c\in \pi_c} \gamma_{l(g_{B_c}) }(\nu_{c|B_c})  \right]  \; ,
  \end{split}
\]
where $l(g_{B_c}) = \lvert B_c \rvert + \Pi(\nu_{c|B_c}) + 2g_{B_c} - 2$, that is \eqref{formula-gen-l}.

\

This completes the proof of Thm.~\ref{thm:asympt-cum-weing}.
 
\qed

\ 

\begin{Cor} 
\label{cor:minimal-l-non-fixed-sigtau}
For $D=1$, consider all the possible $\sigma,\tau$ with fixed numbers of cycles $\#(\sigma)$ and $\#(\tau)$. Consider an arbitrary $k\ge 0$ and a $k$-constellation as in Thm.~\ref{thm:1Nexpansion-Weingarten-Cumulants}, with $l=\#(\sigma) + \#(\tau) -2$. Then $g(\sigma, \tau^{-1})=0$ (as a bipartite map), and $l=\ell(\sigma, \tau)$.
\end{Cor}
\proof 
For $D=1$, denoting $\pi=\Pi(\hat \rho)$,  \eqref{eq:truc} becomes:
\[
  l = \ell(\sigma,\tau)  + 2 \sum_{B\in \pi} g_B +  2 L \big[ \Pi(\sigma, \tau) , \pi ; \Pi(\nu) \big] \;, \;\;   \ell(\sigma,\tau) =  \lVert\sigma \tau^{-1}\rVert +  2 \big( \lvert\Pi(\sigma,\tau)\rvert -1 \big) \; .
\]
The Euler relation of the 2-constellation $(\sigma, \tau^{-1})$ is written  $\lVert \sigma\tau^{-1}\rVert = \#(\sigma) + \#(\tau) + 2g(\sigma, \tau^{-1}) - 2\lvert\Pi(\sigma, \tau)\rvert$, so that the equation above becomes:
\be
\label{eq-l-D1-cycles}
l =   \#(\sigma) + \#(\tau) -2 + 2\Bigl(g(\sigma, \tau^{-1})  +  \sum_{B\in \pi} g_B
+L \big[ \Pi(\sigma, \tau) , \pi ; \Pi(\nu) \big]  \Bigr)\;.
\ee

 \qed

\newpage
\subsection{Monotone Hurwitz numbers and their generalization}
\label{sub:Hurwitz}

In $D=1$, monotone double Hurwitz numbers are obtained by suming the coefficients $p_C[\sigma, \tau; l]$ for $\sigma, \tau$ of fixed cycle types. These numbers have an interpretation in enumerative geometry. We  detail some known facts about monotone Hurwitz numbers, and review the results of our paper in this context. A generalization of monotone double Hurwitz numbers is then introduced, using the coefficients $p_C$ for $D>1$.

\paragraph{Monotone double Hurwitz numbers.} The definition of the number $p_C[\bsig, \btau ; l]$ in Prop.~\ref{def:PC} recalls for $D=1$ the combinatorial definition of monotone Hurwitz numbers. It is known that for $D=1$ and $A$ and $B$ having asymptotic traces of order $N$, the HCIZ integral has the expansion \cite{Goulden1, Goulden2}:
\be
\label{eqref:HCIZ-Hurwitz-double}
 \frac 1 {N^2} \log I_{1, N} (t, A, B) = 
  \sum_{n= 1}^N \frac {t^n}{n!} \sum_{h\ge 0} \frac 1 {N^{2h}} \sum_{\alpha, \beta \vdash n} (-1)^{\#( \alpha) + 
\#(\beta)} \vec H_h(\alpha, \beta) \frac{\Tr_\alpha(A)}
{N^{ \#( \alpha ) }} \frac{\Tr_\beta(B)}{N^{ \#(  \beta )}} + O(t^{N+1})\;,
\ee
where $\alpha, \beta$ are partitions of the integer $n$, $n=\sum_p p\cdot  d_p(\alpha)$, where $d_p(\alpha)$ is the number of parts of $\alpha$ of size $p$, and $\#(\alpha) = \sum_p d_p(\alpha)$ denotes the total number of parts of $\alpha$.

Denoting $C_\alpha$ the set of permutations having fixed cycle type $\alpha$, the coefficients:
\be
\label{eq:def_Hurw_num}
\vec H_{ h }(\alpha, \beta) = \sum_{\sigma \in C_\alpha}\sum_{ \tau \in C_\beta} p_C[\sigma, \tau, l]\;, \qquad 
l= \# (\alpha ) + \# ( \beta ) + 2 h  - 2 \;,
\ee
are the genus-$h$ \emph{monotone double Hurwitz numbers}. In detail these numbers count the number $l+2$-uplets of permutations $\sigma, \tau, \mu_1, \ldots, \mu_l$ with $\sigma\in C_\alpha, \tau\in C_\beta$, and $\mu_i$ transpositions with weakly monotone maxima (Def.~\ref{def:weakly-monotone}) such that $\sigma=\mu_1 \cdots \mu_l \tau$ and furthermore the group generated by $\sigma$ and all the transpositions acts transitively on $\{1,\ldots, n\}$. 

Divided by $n!$, these numbers also count weighted branched coverings of the Riemann sphere by a surface of genus $h$ with $l+2$ branch points, $l$ of which have simple ramifications (that is they have $n-1$ preimages), and the ramifications profiles at 0 and infinity are given respectively by the partitions $\alpha$ and $\beta$ (for more details, see Sec.~\ref{sub:Coverings}). The condition that the transpositions have weakly monotone maxima restricts the admissible coverings.  The formula relating $l$ (the number of simple branch points) and $h$ is the well known Riemann-Hurwitz formula. Note that if $\sigma=\mu_1 \cdots \mu_l \tau$ and the group generated by $\sigma$ and all the transpositions acts transitively on $\{1,\ldots, n\}$, then applying the Euler characteristics formula \eqref{eq:Euler-constellations} to the constellations $(\sigma, \tau^{-1})$, $(\mu_1, \ldots, \mu_l)$, and $(\sigma, \mu_1, \ldots, \mu_l)$ (the last one is connected):
\begin{equation}
h = g(\sigma, \mu_1, \ldots, \mu_l)\; = 
g(\sigma, \tau^{-1})  +  g(\mu_{1}, \ldots, \mu_{l})  + L\bigl[\Pi(\sigma, \tau),  \pi; \Pi(\sigma\tau^{-1})\bigr] \;,
\end{equation}
where $\pi=\lvert \Pi(\mu_{1}, \ldots, \mu_{l})\rvert$, and we recall that $L$ is given by Eq.~\eqref{eq:L}. 

The value of $l$ for $h=0$ is fixed by the Riemann-Hurwitz formula as $l= \#(\alpha) + \#(\beta) - 2 = \#(\sigma) + \#(\tau) - 2 $ for any $(\sigma, \tau) \in C_\alpha\times C_\beta$. From Corollary~\ref{cor:minimal-l-non-fixed-sigtau}, this restricts the sum in Eq.~\eqref{eq:def_Hurw_num} to $\sigma, \tau$ satisfying $g(\sigma, \tau^{-1})=0$ and $l=\ell(\sigma , \tau)$ so that: 
\begin{equation}
\vec H_0(\alpha, \beta) = \sum_{\substack{{(\sigma, \tau) \in C_\alpha\times C_\beta}\\{g(\sigma, \tau^{-1})=0}}} p_C[\sigma, \tau,\ell]\;, \qquad  \ell=\#(\alpha) + \#(\beta) - 2\;,
\end{equation}
and we can use Thm.~\ref{thm:asympt-cum-weing} in $D=1$ (see also \cite{Collins03}) to express $\vec H_0(\alpha, \beta)$ as: 
\be
\label{eq:double-Hurwitz-g0}
\begin{split}
& \vec H_0(\alpha, \beta)  = \crcr
& = \sum_{\substack{{(\sigma, \tau) \in C_\alpha\times C_\beta}\\{g(\sigma, \tau^{-1})=0}}}\;  \prod_{p=1}^n \Biggl(\frac {(2p)!}{p!(p-1)!}\Biggr)^{d_p(\sigma\tau^{-1})} \hspace{-0.6cm}  \sum_{\substack{{\pi\ge \Pi(\nu) }\\{\Pi(\sigma, \tau)\vee\pi = 1_n }\\{\#(\nu) - \lvert\pi\rvert=\lvert\Pi(\sigma, \tau)\rvert - 1}}} \  \prod_{B\in \pi }  \frac{(2\lvert B \rvert + \#({{\nu}_{\lvert_B}}) - 3)!}{(2\lvert B \rvert )!}\;.
\end{split}
\ee
Unlike for single Hurwitz numbers (see \eqref{eq:single-Hurwitz-def} below), one cannot eliminate the sum over permutations, since both the number of connected components and the genus of $(\sigma, \tau^{-1})$ depend on the specific representatives $(\sigma, \tau) \in C_\alpha\times C_\beta$ and not only on their conjugacy classes $\alpha$ and $\beta$.\footnote{For instance, $\bigl((12)(3)(4) \, , \, (12)(3)(4)\bigr)$ and $\bigl((1)(2)(34) \, , \, (12)(3)(4)\bigr)$  have respectively 3 and 2 connected components, while $\bigl((123) \, , \, (321)\bigr)$ and $\bigl((123) \, , \, (312)\bigr)$ have respectively genus 0 and 1.}

\paragraph{Monotone single Hurwitz numbers.} Single Hurwitz numbers are obtained when taking $\tau=\id_n$ (or similarly for $\sigma$) above:
\be
\label{eq:single-Hurwitz-def}
\vec H_h(\alpha) = \sum_{\sigma\in C_\alpha} p_C[\sigma, \id_n ; l] = \lvert C_\alpha \rvert \times p_C[\alpha, \id_n ; l]\;  , \qquad l= \#(\alpha) + n + 2 h - 2\;,
\ee
where we have used the fact that $p_C[\nu, \id; l] $ depends only on the partition $\alpha$ induced by the cycles of $\nu$ and not on the specific representative $\nu\in C_{\alpha}$,
since $m_C(\nu, \id, l, k)$ is invariant under conjugation
(note that we have introduced the somewhat abusive notation $p_C[\alpha, \id_n ; l]$). The cardinal of $C_\alpha$ is
$\lvert C_\alpha \rvert = n! / ( \prod_{p\ge 1}  p^{d_p(\alpha)} d_p(\alpha) ! )  $. Similarly to \eqref{eqref:HCIZ-Hurwitz-double}, these numbers are obtained from the HCIZ integral in $D=1$, but in the case when the asymptotic moments of the matrix $B$ are degenerate, $\lim_{N\rightarrow \infty} \frac 1 N \Tr(B^k)=\delta_{1,k}$ \cite{Goulden2}.

Combining Eq.~\eqref{eq:gamma-and-pC-D1} and 
Eq.~\eqref{eq:single-Hurwitz-def} we get that for any $h \ge 0$ and $\alpha \vdash n$, {\it the genus-$h$ monotone single Hurwitz numbers are, up to signs, the numbers of connected proper constellations with faces $\rho_1\cdots \rho_k =\nu\in C_{\alpha}$, and $\sum_{i=1}^k\lVert\rho_i\rVert = l$ for some $\nu \in C_{\alpha}$:}
\be 
\label{eq:Simple-Hurw-gamma}
\vec H_h(\alpha) =(-1)^{\#(\nu) + n}\; \lvert C_\alpha \rvert \; \gamma_{\#(\nu) + n + 2h - 2}(\nu)\; ,
\ee 
and this is independent of the representative $\nu\in C_\alpha$ chosen. Like the double numbere, they also count 
branched coverings of the sphere and requiring the covering to be of genus 0 comes to requiring that $l= \#(\sigma) + n - 2$, that is its minimal possible value $\ell(\sigma, \mathrm{id})$.
Fixing  $\tau=\id_n$ in \eqref{eq:double-Hurwitz-g0}, as $\Pi(\nu)=\Pi(\sigma, \id)=\Pi(\sigma)$, the sum restricts to $\pi=1_n$ and:
\be
\label{eq:mono-hur-gen-0}
\vec H_0(\alpha) = \frac{n!}{\prod_{p\ge 1} d_p(\alpha) !} \frac{(2n + \#(\alpha) - 3)!}{(2n)!} \prod_{p=1}^n \binom {2p}{p}^{d_p(\alpha)}\;  ,
\ee
reproducing the value found in \cite{Goulden2, Goulden3}. 

The expressions for single Hurwitz numbers \eqref{eq:mono-hur-gen-0} were first obtained in \cite{Collins03} as sums over permutations, that is $p_C$ expressed\footnote{The expressions for $p_C[\sigma, \tau; \ell]$ and $p_C[\sigma, \id; \ell] = \vec H_0(\alpha) /  \lvert C_\alpha \rvert$, where $\alpha$ is the partition of $n$ induced by the disjoint cycles of $\sigma$ are found in Thm.~2.15, respectively  (iii) and (i) of \cite{Collins03}. See also \cite{Harnad}, Eq~(5.41).} as in Thm.~\ref{thm:1Nexpansion-Weingarten-Cumulants}
and then evaluated for zero genus using the counting of planar constellations of Bousquet-Melou -- Schaeffer \cite{BM-Schaeff}. Higher genus monotone single Hurwitz numbers count higher genus constellations, but for now there is no closed simple formula for them. Single and double monotone Hurwitz numbers were later studied in \cite{Goulden1, Goulden2, Goulden3, Goulden4}. To our knowledge, result of \cite{Collins03} leading to Eq.~\eqref{eq:double-Hurwitz-g0} is the only explicit expression of monotone double Hurwitz numbers of genus 0.

\paragraph{Double numbers in terms of the single ones.}  Note that equation Eq.~\eqref{eq:double-Hurwitz-g0} expresses monotone double Hurwitz numbers of genus zero in terms of monotone single Hurwitz numbers:: 
\begin{equation}
\vec H_0(\alpha, \beta)  = \sum_{\substack{{(\sigma, \tau) \in C_\alpha\times C_\beta}\\{g(\sigma, \tau)=0}}}\;  \hspace{-0.4cm}  \sum_{\substack{{\pi\ge \Pi(\nu) }\\{\Pi(\sigma, \tau)\vee\pi = 1_n }\\{\#(\nu) - \lvert\pi\rvert=\lvert\Pi(\sigma, \tau)\rvert - 1}}} \  \prod_{B\in \pi} \frac {\vec H_0(c({\nu}_{B}))} {\lvert C_{c({\nu}_{B})}\rvert}\; ,
\end{equation}
where $\nu=\sigma\tau^{-1}$, and we denoted by $c({\nu}_{B})$ the cycle type of ${\sigma}_{B}{\tau}_{B}^{-1}$ (the associated partition of the integer $\lvert B \rvert$). 
More generally, using Eq.~\eqref{formula-gen-l} for $D=1$ and Eq.~\eqref{eq:Simple-Hurw-gamma}, as well as the fact that if $l$ is given by the Riemann-Hurwitz formula \eqref{eq:def_Hurw_num}, then $\frac{l - \ell(\sigma, \tau)}2 = h - g(\sigma, \tau^{-1}) $  we get:
\begin{Th} 
\label{thm:double-hurwitz-from-single}
For any $h\ge 0$ and $\alpha, \beta \vdash n$, the genus-$h$ monotone double Hurwitz numbers are expressed in terms of the single ones as:
\[
\vec H_h(\alpha, \beta)  = \sum_{\substack{{(\sigma, \tau) \in C_\alpha\times C_\beta}}}\;  \hspace{-0.4cm}  \sum_{L=0}^{h - g(\sigma, \tau^{-1})
} \hspace{-0.4cm}
\sum_{\substack{{\pi\ge \Pi(\nu) }\\{\Pi(\sigma, \tau)\vee\pi = 1_n }\\{\#(\nu) - \lvert\pi\rvert=\lvert\Pi(\sigma, \tau)\rvert + L - 1}}} \ \sum_{ \substack{ \{ g_{B} \}_{B \in \pi }\ge 0     \\ \sum_{B} g_{B}  =  
h - g(\sigma, \tau^{-1})- L} }  \prod_{B\in \pi} \frac {\vec H_{g_B}(c({\nu}_{B}))} {\lvert C_{c({\nu}_{B})}\rvert}\; .
\]
\end{Th}

The sum over partitions can be interpreted as a sum over all ways to add nodes to the 2-constellation $(\sigma, \tau^{-1})$ with a weight per node given by monotone Hurwitz numbers, as explained in \cite{Zub-ZJ} for $h=0$ and in Sec.~\ref{sec:nodal} below for $h>0$.

We thus get an expression for monotone double Hurwitz numbers of genus $h$ for partitions of $n$ in terms of the monotone single Hurwitz numbers of genus $\le h$, for partitions of integers $\le n$ (see e.g.~\cite{Harnad, single-1, single-3, single-2}). We do not know if this relation is already known in the literature.

\paragraph{Higher order monotone Hurwitz numbers.} The tensor generalization of the HCIZ integral naturally gives rise to the following generalization of monotone Hurwitz numbers:
\be
\label{eq:col-hurwitz-def}
\vec H^l(\alpha_1, \beta_1, \ldots, \alpha_D, \beta_D) =  \sum_{ \{ (\sigma_c, \tau_c)\, \in\, C_{\alpha_c}\times C_{\beta_c} \}_{c=1,\dots D}} p_C[\bsig, \btau ; l]\;, 
\ee
where the combinatorial definition of $p_C[\bsig, \btau ; l]$ in terms of factorizations of a $D$-uplet of permutations (Prop.~\ref{def:PC}) generalizes the combinatorial definition of monotone double Hurwitz numbers. Just as in  Thm.~\ref{thm:double-hurwitz-from-single}, $\pC$ can be expressed as a sum over partitions of products of monotone single Hurwitz numbers by using Thm.~\ref{thm:asympt-cum-weing} and \eqref{eq:Simple-Hurw-gamma}.

For each color $c=1,\dots D$ in Prop.~\ref{def:PC}, $p_C[\bsig, \btau ; l]$ counts factorizations of permutations, that is constellations, and thereby certain branched coverings of genera $h_c=g(\sigma_c, \mu_1, \ldots, \mu_{l_c})$, with 
$\lvert\Pi(\sigma_c, \mu_1, \ldots, \mu_{l_c})\rvert$ connected components which satisfy the Riemann-Hurwitz formula for non-connected coverings $l_c= \#(\alpha_c) + \#(\beta_c) + 2 h_c - 2 \lvert\Pi(\sigma_c, \mu_1, \ldots, \mu_{l_c})\rvert $.
However, the transitivity condition in Prop.~\ref{def:PC}, which involves all the permutations, for all $c$, imposes a global constrained on the coverings. We show in Sec.~\ref{sub:Coverings} that 
$\frac 1 {n!}H^l(\alpha_1, \beta_1, \ldots, \alpha_D, \beta_D)$ counts certain weighted \emph{connected} branched coverings of $D$ distinguishable 2-spheres \emph{that ``touch'' at a single common point}. The covering spaces are nodal surfaces, for which a generalization of the genus - the arithmetic genus - is kept fixed.  A generalization of the Riemann-Hurwitz formula relates the arithmetic genus to the number of preimages of the branch points.

\newpage
\section{Interpretation of the combinatorial quantities in terms of nodal surfaces}
\label{sec:nodal}

In this section we give a geometrical picture for the different combinatorial quantities at play: 
\begin{itemize}
\item In Sec.~\ref{sub:Coverings}, we interpret transitive factorizations of multiplets of permutations with conditions on the length of the permutations - like $m_C$ (Thm.~\ref{thm:1Nexpansion-Weingarten-Cumulants}) or $p_C$ (Prop.~\ref{def:PC}) - as branched coverings of $D$ spheres touching at one common point.
\item 
In Sec.~\ref{sub:Geom-descr-pC} we develop step-by-step a geometric understanding of $m_C$, $p_C$, and of the combinatorial formula~\eqref{formula-gen-l} that gives an expression for $p_C$ in terms of single Hurwitz numbers. 
\end{itemize}
Both geometric descriptions involve nodal surfaces, which are collections of surfaces that ``touch'' in groups at certain points called nodes. Let us provide more formal definitions.

\subsection{Nodal surfaces and nodal topological constellations}
\label{sub:def-nodal}

\paragraph{Nodal surfaces.}
Given $p\ge 2$  topological spaces $X_i$, $1\le i\le p$, each with a distinguished point $x_i$, the \emph{wedge sum} of the spaces $X_i$ at the points $x_i$ is the quotient space of the disjoint union of the spaces $X_i$ by the identification $\forall i<j \in \{1,\ldots, p\}$,  $x_i\sim x_j$. A wedge sum of 2-spheres is often poetically called a \emph{bouquet} of 2-spheres. 

A surface is an orientable manifold of dimension two, together with an orientation. Given a surface $X=\sqcup_i X_i$ with $p$ connected components $X_i$, as well as $r$ sets of points $P_j$,  $1\le j \le r$, such that all the elements in all of the sets $P_j$ are distinct  points that may belong to any of the connected surfaces $X_i$ (a set $P_j$ may contain several different points from the same $X_i$), a \emph{nodal surface with $r\ge1$ nodes} is the quotient space of $X$ by the identifications $\forall j\in\{1,\ldots,r\}$, $\forall x, y\in P_j$, $x\sim y$. For each $j\in\{1,\ldots,r\}$, the identification of all the points in $P_j$ defines a \emph{node} or \emph{nodal point}. Such a nodal surface is then denoted by $X^\bullet=X / \{P_j\}_{1\le j \le r}$, and the $X_i$ are said to be its \emph{irreducible components}. 
A wedge sum of surfaces is a nodal surface with one nodal point but the converse is not always true: a surfaces may have two or more distinct points in a node. 

A nodal surface is said to be \emph{connected} if for any two points, there exists a path between them, the path being allowed to jump from one surface to another through a nodal point at which they touch.

A map $F:X^\bullet\rightarrow Y^\bullet$ between two nodal surfaces  is said to be a \emph{homeomorphism} if:
\begin{itemize}
 \item the restriction of $F$ to each irreducible component of $X^\bullet$ is well-defined and is a homeomorphism between surfaces
 \item $F$ preserves the identifications for each node (that is, the points in the irreducible components  of the codomain $Y^\bullet$ that are identified in a given node  are exactly the images of the points that belong to the irreducible components  of the domain $X^\bullet$ that are identified in a node of $X^\bullet$).

\end{itemize}

\paragraph{Arithmetic genus.}The \emph{arithmetic genus} of a connected nodal surface $X^\bullet=\sqcup_{i=1}^p X_i / \{P_j\}_{1\le j \le r}$ with $p$ irreducible components $X_i$ and $r$ nodes is: 
\be 
\label{eq:nodal-genus}
\mathcal{G}(X^\bullet) = \sum_{i=1}^p  g(X_i) + \mathcal L(X^\bullet), \qquad \mathcal L(X^\bullet)=\sum_{j=1}^r\left[\mathrm{Card}(P_j) - 1\right] - p  + 1,
\ee
where $g(X_i)$ is the genus of $X_i$ and $\mathcal L(X^\bullet)$ is the rank of the first homology group of $X^{\bullet}$. This is also the number of excess edges of the abstract graph that has a point vertex for each node $P_j$, a square vertex for each $X_i$, and an edge between a point vertex and a square vertex if the corresponding node belongs to the corresponding irreducible component $X_i$, $P_j\cap X_i \neq\emptyset$.
The arithmetic genus is the genus obtained by ``smoothing'' the nodes, whereas $\sum_{i=1}^p  g(X_i)$ is sometimes called the \emph{geometric genus} of the nodal surface. See the example in Fig.~\ref{fig:arithmetic genus}, which has geometric genus 2 but arithmetic genus 4 ($\mathcal L=2$).
\begin{figure}[!h]
\centering
\includegraphics[scale=0.55]{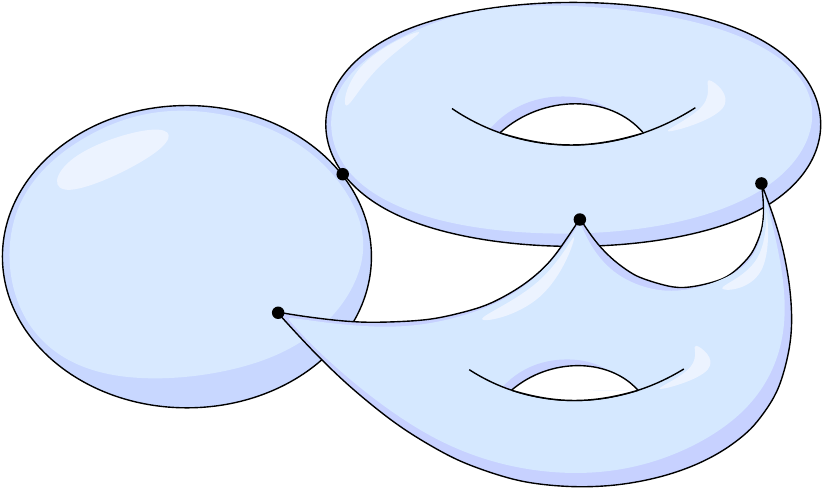} \hspace{2cm} \includegraphics[scale=0.55]{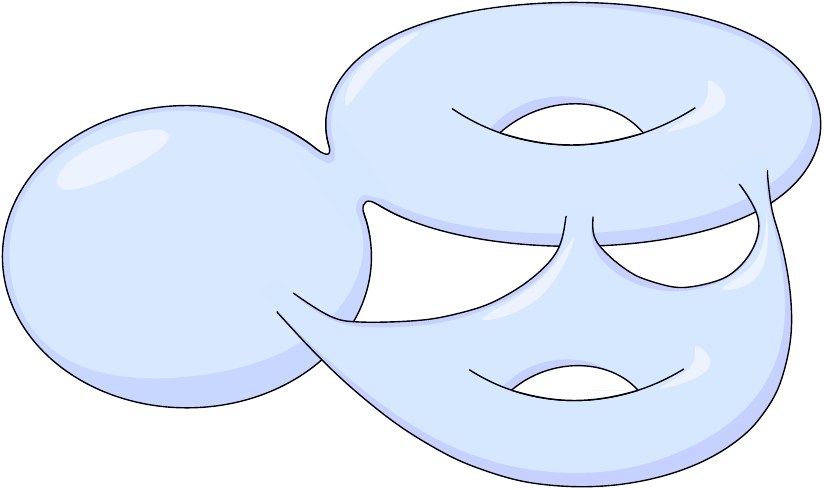} 
\caption{A nodal surface of arithmetic genus 4 (left), and a surface of genus 4 obtained by ``smoothing'' the nodes (right).}
\label{fig:arithmetic genus}
\end{figure}

\paragraph{Nodal topological constellations.} Consider a $D$-uplet $\bethat =(\hat \eta^1, \ldots, \hat \eta^D)$ of constellations defined on the same set of $n$ elements, where $\hat \eta^c = (\eta^c_1, \ldots, \eta^c_{k_c})$ is a $k_c$-constellation, $k_c\ge 1$. As detailed in Sec.~\ref{sec:const}, a graph embedded in a connected surface $X$ is the drawing of a connected graph on $X$ so that the vertices correspond to distinct points on the surface, the images of the edges are paths that may only intersect at the vertices and the complement of the graph in $X$ is homeomorphic to a disjoint union of discs. To simplify the discussion below, we say that a non-connected graph with $p$ components is embedded in a  surface with $p$ connected components if the connected components of the graph are embedded in the connected components of the surface. 

For $X_1, X_2$ two connected surfaces, two embedded graphs $\Gamma_1\subset X_1$ and $\Gamma_2\subset X_2$ are said to be isomorphic if there exists an homeomorphism of surfaces $\phi : X_1 \rightarrow X_2$ whose restriction to $\Gamma_1$ is a graph isomorphism  between $\Gamma_1$ and $\Gamma_2$.  

We consider each constellation $\hat \eta^c$ as an isomorphism class of (non-necessarily connected) embedded graphs (for more details, see \cite{LandoZvonkin}). For each $1\le c \le D$,  and for every choice of graph embeddings $\Gamma_c \subset X_c$ in the isomorphism class, the white vertices are points on the (non-necessarily connected) surface $X_c$, which we denote by $v^c_i$, $1\le i \le n$, and denoting by $P_j=\{v^1_j, \ldots, v^D_j\}$, we consider the nodal surface $\sqcup_{c=1}^D X_c / \{P_j\}_{1\le j \le n}$, together with the graph $\Gamma_c $ embedded in each surface $X_c$. Two such objects, called here \emph{nodal embedded graphs}, are said to be isomorphic if there exists an homeomorphism between the nodal surfaces as defined above, such that the restriction to each domain irreducible component is an isomorphism between embedded graphs. 

We call \emph{nodal topological constellation} the resulting isomorphism classes of nodal embedded graphs. It is uniquely encoded by an ordered multiplet of constellations on the same $n$ elements.

\paragraph{Example: the nodal topological constellation $\cS(\bsig, \btau)$.} We may for instance view $(\bsig, \btau^{-1})$ as a nodal topological constellation, which we denote by $\cS(\bsig, \btau)$, where the role of $\hat \eta^c$ is played by the 2-constellation $(\sigma_c, \tau_c^{-1})$. We represent for each $i\in\{1,\ldots, n\}$  the identification of the white vertices labeled $i$ by introducing a new \emph{triangular} vertex, linked by dotted edges to the white vertices labeled $i$ in every one of the $D$ bipartite maps  $(\sigma_c, \tau_c^{-1})$. 
This is illustrated in Fig.~\ref{fig:nodal-surface-1}.

\begin{figure}[!h]
\centering
\includegraphics[scale=0.7]{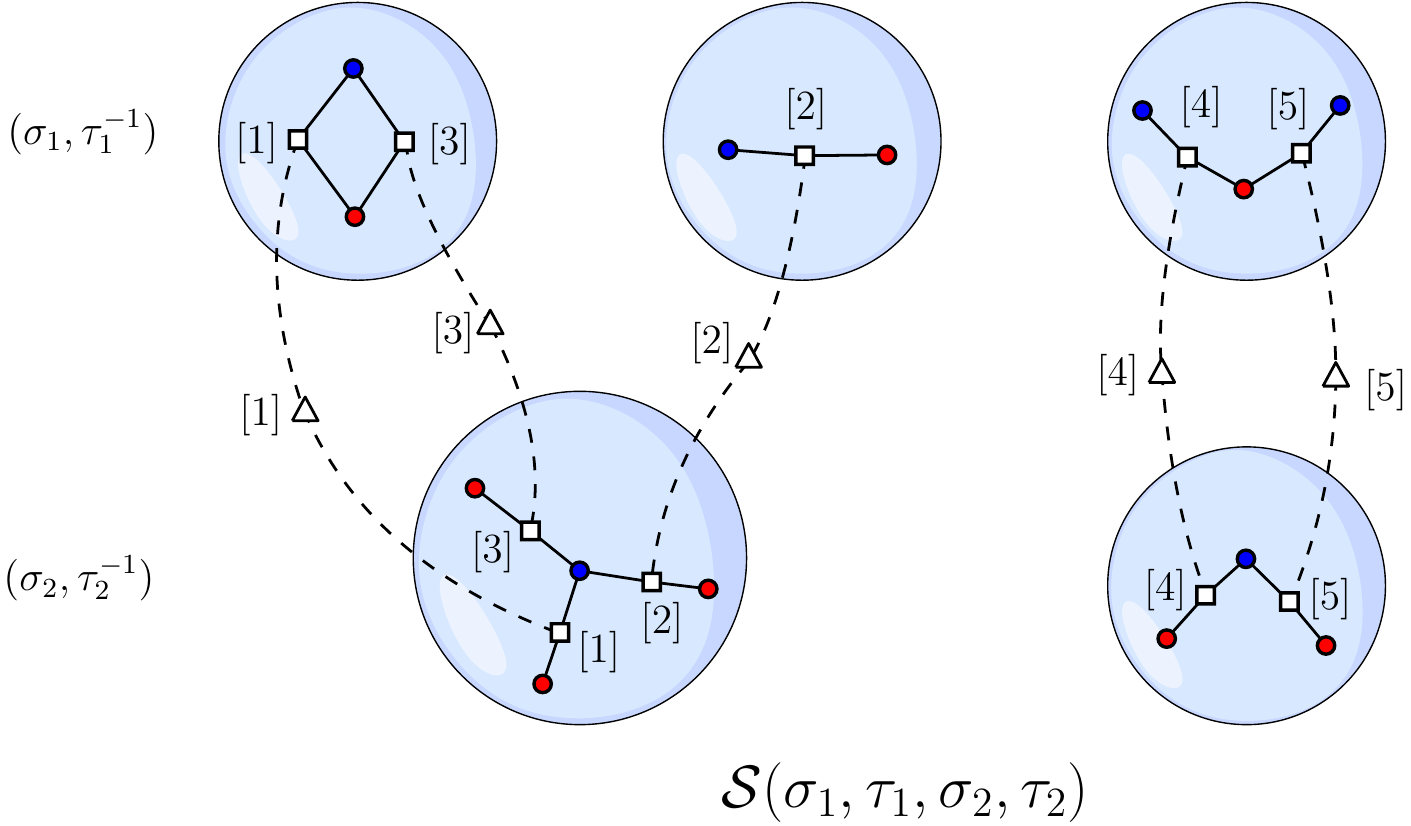} 
\caption{Graphical representation of $\cS(\bsig, \btau)$ for an example in $D=2$, $n=5$, where $\sigma_1= (13)(2)(4)(5)$, $\tau_1^{-1}=(13)(2)(45)$, $\sigma_1\tau_1^{-1}=(1)(2)(3)(45)$, $\sigma_2= (123)(45)$, $\tau_2^{-1}=\id_5$, $\sigma_2\tau_2^{-1}=(123)(45)$. The blue vertices (flavor 1) represent the $\sigma$s and the red vertices (flavor 2) represent the $\tau$s.
Here we have represented the bipartite maps as 2-constellations and added a $D$-valent triangular node for every $i\in\{1, \ldots, n\}$, between the corresponding white vertices (edges of the bipartite map). In this example, $\lvert\Pi(\bsig, \btau)\rvert=2$, and the arithmetic genus is 2.
}
\label{fig:nodal-surface-1}
\end{figure}

\paragraph{Isomorphisms and relabeling.}  Two (topological) $k$-constellations $(\eta_1, \ldots, \eta_{k})$ and $(\rho_1, \ldots, \rho_{k})$ are said to be isomorphic if thy differ by a relabeling of $1,\ldots, n$, that is if there exists $\nu\in S_n$ such that $\eta_i = \nu \rho_i \nu^{-1}$ for all $i$. Two nodal topological constellations encoded respectively by $\bethat$ and $\brho$  where $\hat \eta^c = (\eta^c_1, \ldots, \eta^c_{k_c})$ and $\hat \rho^c = (\rho^c_1, \ldots, \rho^c_{k_c})$  are $k_c$-constellations are said to be isomorphic if there exists $\nu\in S_n$ such that for all $1\le c \le D$ and all $1\le i\le k_c$, $\eta^c_i = \nu \rho^c_i \nu^{-1}$. Note that $\nu$ must be the same for all colors: an isomorphism between nodal constellations is a  \emph{simultaneous} relabeling of $1,\ldots, n$ for all colors.

\paragraph{Transitivity and connectivity.} The introduction of nodal surfaces and nodal topological constellations is motivated by the following lemma: 
\begin{Lem}
\label{Lem:transitivity-nodal}
For $c\in\{1,\ldots,D\}$, let $\hat \eta^c = (\eta^c_1, \ldots, \eta^c_{k_c})$ be a $k_c$-constellation, $k_c\ge 1$. The number of transitivity classes $\lvert \Pi(\bethat)\rvert$ of the group generated by all $\hat \eta_i^c$ on $\{1,\ldots, n\}$  is the number of connected components of the corresponding nodal topological constellation.
\end{Lem}

For instance, $\cS(\bsig, \btau)$ is connected if and only if the group generated by all $\sigma_c, \tau_c$ acts transitively on $\{1,\ldots, n\}$, and more generally, the number of connected components of this nodal topological constellation is $\lvert\Pi(\bsig, \btau)\rvert$.

\proof  Consider a representative in the isomorphism class of nodal embedded graphs, that is, a nodal surface $X^\bullet=\sqcup_{c}  X_c / \{P_j\}_{1\le j \le n}$  together with the graph $\Gamma_c$ corresponding to $\hat \eta^c$ embedded in the surface $X_c$ for every $c$. From the definition of an embedded graph, we know that $X_c$ has $\lvert \Pi(\hat \eta^c)\rvert$ connected components. 
It is therefore enough to show that two elements $a,b\in\{1,\ldots,n\}$ are in the  same transitivity class of the group generated by $\{\eta^c_1, \ldots, \eta^c_{k_c}\}_{1\le c\le D}$ if and only if there exists a path between the corresponding nodes on the graph $\Gamma^\bullet$ obtained from $\Gamma_1, \ldots, \Gamma_D$ by identifying the $D$ white vertices of flavor $k$ for each $k\in\{1,\ldots,n\}$. 

Two elements $a,b\in\{1,\ldots,n\}$ are in the same transitivity class of the group generated by $\{\eta^c_1, \ldots, \eta^c_{k_c}\}_{1\le c\le D}$ if and only if there exists a word  $w$ in these permutations and their inverses so that $w(a)=b$. Assuming that this is the case, we may build a path between the points corresponding to the two nodes labeled $a$ and $b$ in $\{1,\ldots, n\}$ in $\Gamma^\bullet$  as follows: we read the word $w$ from right to left,  when encountering a permutation $\eta^c_{i}(d)$, $1\le d \le n$ the path follows the two edges of flavor $i$ from the node labeled $d$ to the node labeled $\eta^c_{i}(d)$ on the embedded graph $\Gamma_c \subset X_c$, and similarly for $\eta^c_{i}(d)^{-1}$. \emph{The important point is that there is no problem  in successively applying permutations of different colors, since the path may go between any two $\Gamma_c\subset X_c$ and $\Gamma_{c'}\subset X_{c'}$ at any node.} Conversely, a path in $\Gamma^\bullet$ from a node $a$ to a node $b$ is composed of successive steps from a node to a vertex of flavor $i$ via an edge $e$ and on to another node $j$ via an edge $e'$ is some $\Gamma_c$ for some $c$. 
To each such step we associate the permutation $(\eta_i^c)^{(d+1)}$, $d$ being the number of edges encountered when turning from $e$ to $e'$ around the flavored vertex clockwise.  A word $w$ such that $w(a)=b$ is then obtained by composing these permutations from right to left.
\qed



\

Now, as the number of connected components of the nodal topological constellation is $\lvert \Pi (\bethat)\rvert$, the number of its irreducible components is $\sum_c \lvert \Pi(\hat \eta^c)\rvert$, its geometric genus is $\sum_c g(\hat \eta^c)$, and the surface has $n$ nodal points each with cardinal $D$, we obtain the arithmetic genus
\eqref{eq:nodal-genus} of the nodal topological constellation:
\be
\label{eq:nodal-genus2}
\mathcal{G} (\bethat) =
 \sum_{c=1}^D g(\hat \eta^c)  +
 D n - ( \sum_{c=1}^D \lvert \Pi(\hat \eta^c)\rvert + n     -  \lvert \Pi (\bethat)\rvert) = 
\sum_{c=1}^D \left( g(\hat \eta^c) - \lvert \Pi(\hat \eta^c)\rvert \right)+ n(D-1)+ \lvert \Pi (\bethat)\rvert.
\ee

For instance, combining this formula for $\cS(\bsig, \btau)$, with Eq.~\eqref{eq:Euler-constellations-lengths} for 
the Euler characteristics of $(\sigma_c, \tau_c^{-1})$ we get:
\be 
\label{eq:Eulersigmatau}
2  g(\sigma_c, \tau_c^{-1}) - 2 \lvert \Pi(\sigma_c, \tau_c^{-1})\rvert  = \lVert \sigma_c \tau_c^{-1}\rVert - \#(\sigma_c) - \#(\tau_c) \;.
\ee
Supplementing this by the definition \eqref{eq:l-minimal} $\ell(\bsig, \btau)=\sum_{c=1}^D \lVert\sigma_c\tau_c^{-1}\rVert  + 2  \big(\lvert \Pi(\bsig,\btau)\rvert-1 \big)$ leads to the following lemma.
\begin{Lem} 
The arithmetic genus $\mathcal G(\bsig, \btau)$ of $\cS(\bsig, \btau)$ is related to $\ell(\bsig, \btau)$  by:
\be
\label{eq:relation-genus-ell-1}
\ell(\bsig, \btau) = \sum_{c=1}^D\left[ \#(\sigma_c) + \#(\tau_c)\right] + 2 \mathcal G(\bsig, \btau) - 2 - 2n(D-1).
\ee
\end{Lem}

\subsection{Transitive factorizations of multiplets of permutations and branched coverings of a bouquet of 2-spheres}
\label{sub:Coverings}

Given two topological spaces $X$ and $Y$, and a subset $L$ of $Y$, a map $f:X\rightarrow Y$ is said to be an \emph{$n$-sheeted branched covering of $Y$ branched over $L$}, if $f$ restricted to the complement of the preimage of $L$ in $X$ is continuous, and such that for every $y\in Y\setminus L$, there exists an open neighborhood $U$ such that $f^{-1}(U)$ is homeomorphic to $U\times \{1, \ldots, n\}$.  Two branched coverings $f_1:X_1 \rightarrow Y$ and $f_2:X_2 \rightarrow Y$ are said to be \emph{isomorphic} if there exists an orientation preserving homeomorphism $u:X_1 \rightarrow X_2$ such that $f_1=f_2 \circ u$.  The set $L$ is called the \emph{branch locus}, $Y$ the \emph{target space}, and $X$ the \emph{covering space}. The number of connected components of a covering is that of the covering space. For $X,Y$ two nodal surfaces, $L$ consists of points called \emph{branch points}, and their preimages are called \emph{singular points}. 

\

It is well known (see \cite{LandoZvonkin}) that $n$-sheeted coverings of the oriented 2-sphere branched over $k$ ordered points up to isomorphisms are in one-to-one correspondence with $k$-constellations $\hat \eta = (\eta_1, \ldots, \eta_{k})$, that is $k$-uplets of permutations of $n$ elements, such that $\eta_1\cdots\eta_k = \id$, up to isomorphisms. Given such an unlabelled constellation, an isomorphism class of branched coverings is obtained by sending each face of the corresponding topological constellation (for every surface in the isomorphism class) to the face of the unique constellation with one white vertex. Each ``star'' in the constellation formed by a white vertex and its incident edges thus corresponds to the preimage of the only ``star'' in the target space. The  vertices with flavors of the constellation correspond to the  singular points, and the partitions of $n$ that label the conjugacy classes of the permutations $\eta_1, \ldots, \eta_{k}$, called \emph{ramification profiles}, describe the way in which the $n$ sheets meet in groups at the singular points.  The covering space is a collection of $K=\lvert \Pi(\hat \eta)\rvert$ connected surfaces seen up to isomorphisms, whose genera sum up to $h=g(\hat \eta)$. The \emph{Riemann-Hurwitz formula} relates these two numbers: 
\be
\label{eq:RiemHurw}
\sum_{i=1}^k \lVert \eta_i\rVert= 2(n + h - K)\;,
\ee
where for the branch point labeled $i$, $ \lVert \eta_i\rVert =  n-\#( \eta_i)$ is the  difference between the number $n$ of preimages that the point would have if it was not in the branch locus, and the number of preimages it actually has.

A $D$-uplet of constellations $\bethat=(\hat \eta^1, \ldots, \hat \eta^D)$, where $\hat \eta^c = (\eta^c_1, \ldots, \eta^c_{k_c})$ is a $k_c$-constellation on $n$ elements, $k_c\ge 1$, up to isomorphisms is therefore in bijection with $D$ branched coverings $f_1, \ldots f_D$ of the 2-sphere $S$, up to isomorphisms, $f_c : X_c \rightarrow S$ being branched  over $k_c$ points. Unlike for nodal constellations, here the isomorphisms are for each color independently, that is, independent relabelings of $1,\ldots,n$ for different colors are allowed.
There is no direct interpretation in this context for the quantity $\lvert \Pi(\bethat)\rvert$, which moreover is not invariant under relabelings of $\{1,\ldots,n\}$ for each color independently: it is only invariant under simultaneous relabelings for all colors. 
On the other hand, $\lvert \Pi(\bethat)\rvert$ has a natural interpretation in the context of nodal constellations, as stated in the following theorem.

\begin{Th}
\label{th:bij-coverings-topoCons}
Isomorphism classes of connected branched coverings of a bouquet of $D$ distinguishable 2-spheres $S_c$, $c\in \{1,\ldots,D\}$ branched over a set of $k+2D$ ordered points that do not belong to the nodal point, $k_c+2$ of which belong to $S_c$ for each $c$ ($\sum_{c=1}^D k_c =k$) are in one-to-one correspondence with  systems of permutations of the type: 
\begin{itemize}
   \item For $c\in\{1, \ldots,D\}$, $\eta^c_0, \ldots, \eta^c_{k_c+1} \in S_n$ such that  we have  $\mathrm{id}=\eta^c_0\cdots \eta^c_{k_c+1}$,
    \item $\lvert\Pi(\bethat)\rvert=1$, that is, the group generated by all the permutations is transitive on
    \{1,\ldots, n\},
\end{itemize}
up to isomorphisms of nodal constellations (up to simultaneous relabeling of $1,\ldots,n$ for all $1\le c \le D$).
\end{Th}

\proof We prove the correspondence between topological objects, knowing the correspondence between nodal topological constellations and systems of permutations. Consider a branched covering $f:X^\bullet\rightarrow Y^\bullet$ where $Y^\bullet$ is a bouquet of $D$ distinguishable 2-spheres $S_c$, $c\in \{1,\ldots,D\}$. On each 2-sphere $S_c$ of the target space $Y^\bullet$, one can draw a star-graph $\gamma_c$ by adding non-crossing arcs between the $k_c$ branch points and the nodal point so that the order of the arcs around the nodal point grows from 1 to $k_c$ clockwise (see the right of Fig.~\ref{fig:covering-nodal-surface-2}). Doing this for all $c$, we get a nodal embedded graph $\gamma^\bullet\subset Y^\bullet$, whose preimage $\Gamma^\bullet\subset X^\bullet$  is a representative of a nodal topological constellation in the sense that it is a representative in the corresponding isomorphism class of nodal embedded graphs. There is no labeling of the $n$ preimages of the nodal point, so that the nodal constellation can be seen up to isomorphisms (up to simultaneous relabelings of $1,\ldots,n$ for all $1\le c \le D$). 


For two isomorphic branched coverings $f_1:X_1^\bullet \rightarrow Y^\bullet$ and $f_2:X_2^\bullet \rightarrow Y^\bullet$, there exists by definition an orientation preserving homeomorphism of nodal surfaces  $u:X_1^\bullet \rightarrow X_2^\bullet$ such that $f_1=f_2 \circ u$. Denoting by $\Gamma_1^\bullet\subset X_1^\bullet$ and $\Gamma_2^\bullet\subset X_2^\bullet$ the preimages of  $\gamma^\bullet\subset Y^\bullet$  by $f_1$ and $f_2$ respectively, it is clear that $u$ restricted to each irreducible component of $X_1^\bullet$ is an isomorphism between embedded graphs: it is an homeomorphism of surfaces by definition, and it is clear that the restriction of $u$ to $\Gamma_1^\bullet$ on each irreducible component is a graph isomorphism.  Therefore, the nodal embedded graphs $\Gamma_1^\bullet\subset X_1^\bullet$ and $\Gamma_2^\bullet\subset X_2^\bullet$ are isomorphic, and are two representatives of the unlabelled nodal topological constellation.

This defines a map from isomorphisms classes of branched coverings of $Y^\bullet$ to isomorphism classes of nodal topological constellations, and we now verify that this map is invertible. Indeed, consider a representative $\Gamma^\bullet\subset X^\bullet$ of a nodal topological constellation $\bethat=(\hat \eta^1, \ldots, \hat \eta^D)$, where $\hat \eta^c = (\eta^c_1, \ldots, \eta^c_{k_c})$ is a $k_c$-constellation. $\Gamma^\bullet\subset X^\bullet$  is a nodal embedded graph, and for every $c$, we denote by $X_c$ the disjoint union of the irreducible components of $X^\bullet$ that contain vertices associated with $\eta^c_0$. A branched covering $f:X^\bullet\rightarrow Y^\bullet$ is then obtained by choosing homeomorphisms sending each connected component of the complement of the graph $\Gamma^\bullet$ in $X_c$ to the complement of the star-graph $\gamma_c$ in the irreducible component $S_c$ of $Y^\bullet$. 

Given two representatives $\Gamma_1^\bullet\subset X_1^\bullet$ and $\Gamma_2^\bullet\subset X_2^\bullet$  of a nodal topological constellation, there exists an homeomorphism of nodal surfaces $u:X_1^\bullet \rightarrow X_2^\bullet$ that induces an isomorphism of embedded graphs on every irreducible component of $X_1^\bullet$. Considering the branched coverings $f_1:X_1^\bullet\rightarrow Y^\bullet$ and $f_2:X_2^\bullet\rightarrow Y^\bullet$ constructed as in the previous paragraph, we see that $f_1=f_2 \circ u$  so that $f_1$ and $f_2$ are isomorphic.

The construction described above that associates a covering $f$ to a representative $\Gamma^\bullet\subset X^\bullet$ is independent of the labeling of the nodal points of $\Gamma^\bullet$, so that we have defined the converse map from isomorphism classes of nodal topological constellations to isomorphisms classes of branched coverings of $Y^\bullet$.

The statement regarding the number of connected components is a direct consequence of Lemma~\ref{Lem:transitivity-nodal}.

\qed 

\

We illustrate this for the following example with $n=5$, $D=2$, $k_1=2$, $k_2=3$, for which the nodal topological constellation is represented graphically on the left of Fig.~\ref{fig:covering-nodal-surface-2}: 
For the permutations of $\hat \eta^1$: $\eta_0^1= (12)(3)(4)(5)$, $\eta_1^1=(12)(34)(5)$, $\eta^1_{2}=(12)(345)$, $\eta_1^{3}=(12)(3)(45)$;  For the permutations of $\hat \eta^2$:
$\eta_0^2= (132)(45)$, $\eta^2_{1}=(15)(24)(3)$, $\eta^2_{2}=(1)(23)(4)(5)$, $\eta^2_{3}=(14)(2)(35)$, $\eta_2^{4}=\id_5$. 
In the figure, the color representing the flavors 0,1,2,3,4 are in that order pink, blue, red, orange, green. 
The nodal topological constellation is connected and has arithmetic genus 4. 
The fact that $\eta_2^{4}=\id_5$ (all green vertices are leaves in the nodal constellation) means that when interpreted as a branched covering of a bouquet of two 2-spheres, the green vertex in the target space is actually not a branch point since it has 5 preimages.

\begin{figure}[!h]
\centering
\includegraphics[scale=0.7]{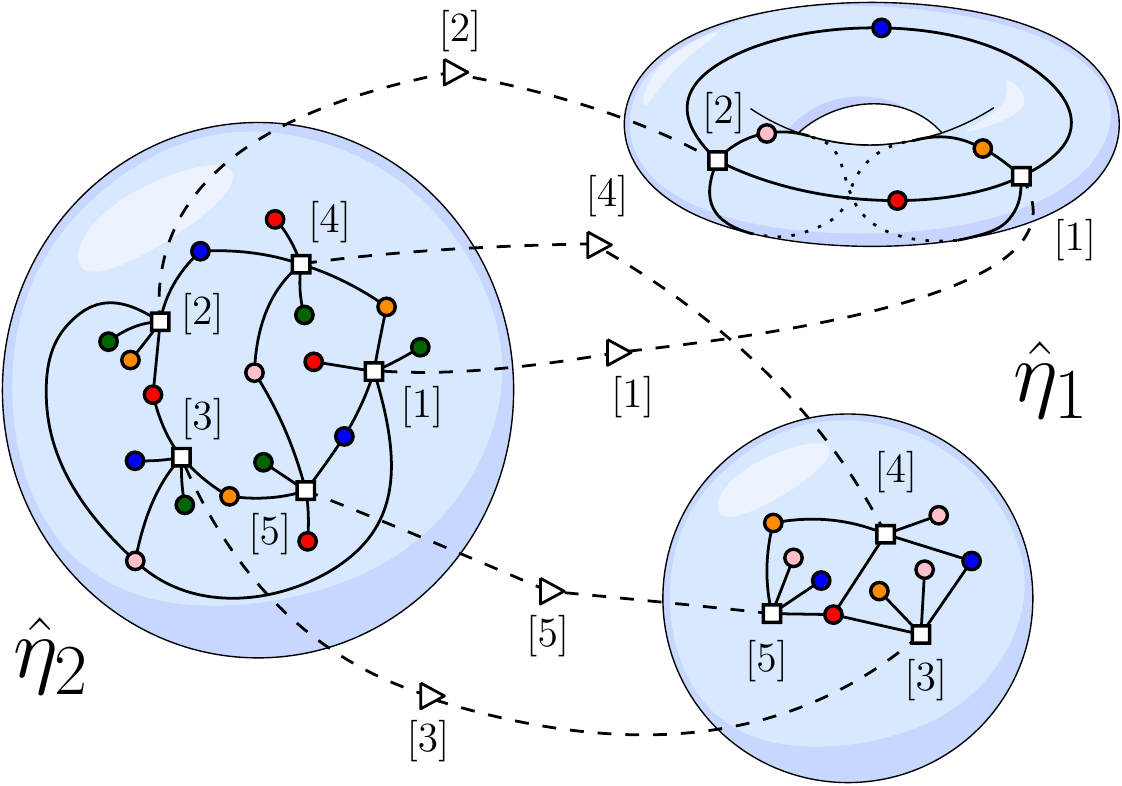} \qquad\raisebox{+15ex}{ {\Huge $\substack{{\rightarrow}\\{\scriptscriptstyle f}}$}}\qquad \raisebox{+5ex}{\includegraphics[scale=0.6]{target.pdf}}
\caption{A nodal topological constellation (left) can be interpreted, up to relabeling of the nodes, as as an isomorphism class of branched coverings $f$ of a bouquet of $D$ distinguishable 2-spheres (right). In this example, $D=2$, $n=5$, $k_1=2$, and $k_2=3$.}
\label{fig:covering-nodal-surface-2}
\end{figure}

Fixing $D,k\ge 1$, $\mathcal H \ge 0$ and for all $c\in\{1,\ldots,D\}$, $\alpha_c, \beta_c \vdash n$ non-trivial, we let $\mathfrak{C}_{\mathcal H}[\{\alpha_c, \beta_c\}_c, k]$ be the set of isomorphism classes of connected $n$-sheeted branched coverings of a bouquet of $D$ distinguishable 2-spheres $S_c$, $c\in \{1,\ldots,D\}$ branched over a set of precisely\footnote{That is, every one of these points has less than $n$ preimages. Since $\alpha_c, \beta_c$ are non-trivial and since the other permutation involved in the definitions of $m_C$ and $p_C$ are different from the identity permutation (proper), all the $k+2D$ points have non-trivial ramifications.} $k+2D$ ordered points that do not belong to the nodal point, at least two of which belong to $S_c$ for each $c$, so that the first and last points for each $c$ respectively have ramification profiles $\alpha_c$ and $\beta_c$\footnote{For the Riemann sphere, these points are usually taken to be at zero and infinity.}, and so that the 
arithmetic genus of the covering space as defined in \eqref{eq:nodal-genus}  is $\mathcal H$. 

Note that for an element $X$ of $\mathfrak{C}_{\mathcal H}[\{\alpha_c, \beta_c\}_c, k]$,  the nodal surfaces in the isomorphism class have $n$ nodal points (the node of the bouquet of spheres does not belong to the branch locus and its $n$ preimages are the only nodes of $X$), so that $\mathcal L(X) =n(D-1) - p  + 1$ in \eqref{eq:nodal-genus}.

  We let  $\mathfrak{B}_{\mathcal H}[\{\alpha_c, \beta_c\}_c, k]$ be the subset of $\mathfrak{C}_{\mathcal H}[\{\alpha_c, \beta_c\}_c, k]$ of the elements $X$ for which the branch points whose ramification profiles are not fixed to one of the $\alpha_c$ or $\beta_c$ have simple ramification (they have $n-1$ preimages) and satisfy an additional monotonicity condition: Consider any set of permutations encoding $X$ (a choice of labeling of $1,...,n$ in Th.~\ref{th:bij-coverings-topoCons}). For $c\in\{1,\ldots,D\}$, the transpositions encoding the ramification profiles of the branch points in $S_c$ whose ramification profiles are not fixed to $\alpha_c$ or $\beta_c$ inherit an ordering from the global ordering of the branch points. With this ordering, these transpositions must  have weakly monotone maxima. 
  
  We recall  that for $\alpha\vdash n$, the conjugacy class $C_\alpha$ gathers the permutations in $S_n$ whose cycle-type is $\alpha$, that $m_C$ and $p_C$  were respectively defined in  Thm.~\ref{thm:1Nexpansion-Weingarten-Cumulants} and Prop.~\ref{def:PC}, as well as the definition of higher order monotone double Hurwitz numbers \eqref{eq:col-hurwitz-def}:
  \be
\label{eq:col-hurwitz-def-2-1}
\vec H^l(\{\alpha_c, \beta_c\}_c) = \sum_{c=1}^D \, \sum_{(\sigma_c, \tau_c)\, \in\, C_{\alpha_c}\times C_{\beta_c}} p_C[\bsig, \btau ; l]\;. 
\ee
 We also define the following generalization of the Bousquet-Melou--Schaeffer numbers \cite{BM-Schaeff}:
   \be
\label{eq:col-hurwitz-def-2}
\vec {BS}^l_k(\{\alpha_c, \beta_c\}_c) = \sum_{c=1}^D \, \sum_{(\sigma_c, \tau_c)\, \in\, C_{\alpha_c}\times C_{\beta_c}} m_C(\bsig, \btau ; l,k)\;. 
\ee

\begin{Cor} 
\label{cor:Gen-Hurwitz-nodal}
Fixing $D,k\ge 1$, $\mathcal H \ge 0$, and for all $c\in\{1,\ldots,D\}$, $\alpha_c, \beta_c \vdash n$ non-trivial and  $\sigma_c \in C_{\alpha_c}$, $\tau_c^{-1} \in C_{\beta_c}$, and defining:
\be
\label{eq:RiemHurw-nodal}
l  = \sum_c \left[ \#(\alpha_c) + \#(\beta_c)\right]  + 2\mathcal H  - 2 - 2n(D-1),
\ee
\begin{enumerate}
\item The cardinal of  $\mathfrak{C}_{\mathcal H}[\{\alpha_c, \beta_c\}_c, k]$ is $\frac 1 {n!}\vec {BS}^l_k(\{\alpha_c, \beta_c\}_c)$.
\item The cardinal of $\mathfrak{B}_{\mathcal H}[\{\alpha_c, \beta_c\}_c, k]$  is $\frac 1 {n!} \vec H^k(\{\alpha_c, \beta_c\}_c)$ and $k=l$ as defined in \eqref{eq:RiemHurw-nodal}.
\end{enumerate}
In both cases, the total number of singular points is given by:
\be 
nk-l + \sum_{c}\#(\alpha_c) + \#(\beta_c)  = 2 - 2\mathcal{H} + n(k+D-1).
\ee
\end{Cor}

The relation \eqref{eq:RiemHurw-nodal} should be compared to the Riemann-Hurwitz formula \eqref{eq:def_Hurw_num}.

\proof Let $f:X^\bullet\rightarrow Y^\bullet$ be an element of $\mathfrak{C}_{\mathcal H}[\{\alpha_c, \beta_c\}_c, k]$, where $Y^\bullet$ is a bouquet of $D$ distinguishable 2-spheres $S_c$, $c\in \{1,\ldots,D\}$. Then there exists $k_1, \ldots, k_D \ge 0$ such that $\sum_c k_c =k$ and for each $c$, $k_c+2$ of the (ordered) branch points belong to $S_c$, and the first and last respectively have ramification profiles $\alpha_c$ and $\beta_c$. From Thm.~\ref{th:bij-coverings-topoCons} and its proof, $f$ is bijectively mapped to $D$ ordered sequences of  permutations:    $$\hat \eta^c = (\sigma_c^{-1}, \eta^c_1, \ldots, \eta^c_{k_c}, \tau_c) \in S_n^{k_c+2},\quad \mathrm{s.t.}\quad \mathrm{id}=\sigma_c^{-1}\eta^c_1\cdots \eta^c_{k_c}\tau_c,$$  $c\in\{1, \ldots,D\}$, such that $\sigma_c\in C_{\alpha_c}$ and $\tau_c\in C_{\beta_c}$, and  the group generated by all the permutations is transitive on $\{1,\ldots, n\}$, up to simultaneous relabelings of $1,\ldots,n$ for all $c$, and so that the arithmetic genus of the nodal topological constellation encoded by this system of permutations is $\mathcal H$. This explains that  $\mathrm{Card}\mathfrak{C}_{\mathcal H}[\{\alpha_c, \beta_c\}_c, k]= \frac 1 {n!}\vec {BS}^l_k(\{\alpha_c, \beta_c\}_c)$ if we show that $\sum_{c=1}^D\sum_{i=1}^{k_c} \lVert \eta_i^c\rVert$ is given by the right hand side of \eqref{eq:RiemHurw-nodal}.
    
    From \eqref{eq:nodal-genus2}, 
$\mathcal H = \sum_{c=1}^D \left( g(\hat \eta^c) - \lvert \Pi(\hat \eta^c)\rvert \right)+ n(D-1)+ 1$, and from the Euler characteristics \eqref{eq:Euler-constellations-lengths} of $\hat \eta^c$, $2\left(g(\hat \eta^c) - \lvert \Pi(\hat \eta^c)\rvert \right) = \sum_{i=1}^{k_c} \lVert \eta_i^c\rVert - \#(\sigma_c) - \#(\tau_c)$, so that:
$$
\sum_{c=1}^D\sum_{i=1}^{k_c} \lVert \eta_i^c\rVert = \sum_{c=1}^D\left[ \#(\sigma_c) + \#(\tau_c) \right] + 2 (\mathcal H -n(D-1)  - 1 ),
$$
which proves the first point of the corollary. For the elements of $\mathfrak{B}_{\mathcal H}[\{\alpha_c, \beta_c\}_c, k]$, $\eta^c_1, \ldots, \eta^c_{k_c}$ are transpositions with weakly monotone maxima, and the total number $k$ of these transposition is also $\sum_{c=1}^D\sum_{i=1}^{k_c} \lVert \eta_i^c\rVert$. This concludes the proof.

\qed

\begin{Rk} For the case $D=2$, the enumeration of isomorphism classes of branched covers of a bouquet of two 2-spheres should be relevant in the context of compactifications of the moduli spaces of curves such as the Deligne-Mumford compactification, where the necessity to include degenerated cycles  implies considering nodal surfaces where at each node only two surfaces meet \cite{CountingSurfaces, CavalieriMiles, Zvonkine, LandoZvonkine2}. 
\end{Rk}

\begin{Rk} We have presented a geometrical interpretation based on nodal surfaces. From the colored structure, the reader familiar with the literature on colored triangulations and random tensor models will recognize a combinatorial encoding that recalls that of colored triangulations in dimension two and higher. This begs for an interpretation in terms of higher dimensional objects, instead of nodal surfaces, but we leave this for future work.
\end{Rk}

\subsection{The $1/N$ expansions as topological expansions
}
\label{sub:Geom-descr-pC}

The aim of this subsection is to provide a combinatorial and geometric interpretation to the formulas of Theorem \ref{thm:asympt-cum-weing}. The transpositions $\mu^c_i$ in the combinatorial definition of $p_C$ (Prop.~\ref{def:PC}) do not appear for instance  in \eqref{formula-gen-l}: the intuition is that we should try to keep all the $\sigma_c, \tau_c$ fixed on one hand, and ``resum'' the contributions of all the $\mu^c_i$ on the other hand, in some way. 
To this aim, given $D$ sequences of permutations  $\sigma_c, \tau_c, \eta^c_1, \ldots, \eta^c_{k_c}$ for $c\in\{1,\ldots, D\}$ such that $\id=\sigma_c^{-1} \eta^c_1 \cdots \eta^c_{k_c} \tau_c$, instead of considering the  nodal topological constellation encoded by the $(\sigma_c^{-1},  \eta^c_1,  \ldots,  \eta^c_{k_c},  \tau_c)$ for all $c$ as in Corollary~\ref{cor:Gen-Hurwitz-nodal}, we will rather consider a new kind of isomorphism class of nodal surfaces from the nodal topological constellation $\cS(\bsig, \btau)$ on one hand, and the $D$ topological constellations $\hat \eta^c=(\eta^c_1,  \ldots,  \eta^c_{k_c})$ on the other hand.

\subsubsection{Nodal surfaces for (\TitleGsig, \TitleGtau, $ \bethat$)
} 

We fix  $\bsig, \btau\in \bS_n$ as well as $k,l\ge 0$, and for $c\in \{1,\ldots, D\}$, we let $\hat \eta^c=(\eta^c_1, \ldots, \eta^c_{k_c})$ be a $k_c$-constellation, $k_c\ge 0$, such that $\sum_{c=1}^Dk_c =k $  and:
\be
\label{eq:colorwise-constellation-bis}
\forall c  \; , \;\;  \sigma_c\tau_c^{-1}=\eta^c_{1}\cdots \eta^c_{k_c} \;.
\ee
subject to the conditions:
\begin{enumerate}[label=(C\arabic*)]
\item the collection of all $\{\hat\eta^c,\tau_c\}_c$  acts transitively on $\{1,\ldots , n\}$,
\item $\sum_{c=1}^D\sum_{i_c=1}^{k_c} \lVert\eta^c_{i_c}\rVert =l$. 
\end{enumerate}
This data defines:
\begin{itemize}
    \item  a (non-necessarily connected) nodal topological constellation $\cS(\bsig, \btau)$ as defined in Sec.~\ref{sub:def-nodal} (see Fig.~\ref{fig:nodal-surface-1}),
   \item a (non-necessarily connected) topological $k_c$-constellation $\hat \eta^c$ for each $c\in\{1, \ldots, D\}$. 
\end{itemize}

Since $ \nu_c=\sigma_c\tau_c^{-1}=\eta^c_{1}\cdots \eta^c_{k_c}$, the disjoint cycles  of $\sigma_c\tau_c^{-1}$ and the disjoint cycles of $\eta^c_{1}\cdots \eta^c_{k_c}$ match, so that for every nodal embedded graph $\Gamma^\bullet \subset X^\bullet$ in the isomorphism class $\cS(\bsig, \btau)$ and every embedded graphs $\Gamma_1\subset Y_1, \ldots,\Gamma_D\subset  Y_D$ in the isomorphisms classes $\hat \eta^1, \ldots, \hat \eta^D$, there is a one-to-one correspondence $\Psi$ between the faces $F_1, \ldots, F_{nk-l}$ of $\Gamma^\bullet \subset X^\bullet$ (the connected components of the complement of the graph $\Gamma^\bullet$ in the nodal surface $ X^\bullet$), and the faces  $F'_1, \ldots, F'_{nk-l}$ of the $\Gamma_c\subset Y_c$ for $c\in\{1, \ldots, D\}$, where the labelings are chosen so that 
$\Psi(F_j) = F'_j$. To render this pairwise identification obvious, we choose for each $j\in\{1,\ldots, nk-l\}$ two points $v_j$ and $v'_j$ respectively in the interiors of $F_j$ and  $F'_j$, and we consider the nodal surface $Z^\bullet= (X^\bullet \sqcup_{c=1}^D Y_c) / \{P_j\}_j  $ (together with the graphs $\Gamma^\bullet$ and $\Gamma_c$ drawn on $Z^\bullet$).

We then call $\cS(\bsig, \btau, \bethat)$ the isomorphism class of such objects, where by isomorphisms we mean the homeomorphisms of nodal surfaces that induce an isomorphism of embedded graph on each irreducible component, and preserve the incidence between the nodes $P_j$ and the faces $F_j$ and $F'_j$, in the sense that if a node $P_j$ belongs to the interior of the faces $F_j$ and $F'_j$, then the image of the node also belongs to  the interior of the images of the faces. 

An example is shown in Fig.~\ref{fig:big-nodal-surface}, where the nodal points in the interior of the faces are represented by dotted edges (whereas we recall that the nodes of $\cS(\bsig, \btau)$ are represented by dashed edges linking triangular vertices).

\ 

In this context, the graph $ G \big[\Pi(\bsig, \btau) , \{\pi_c\}_c ; \{ \Pi(\nu_c) \}_c \big] $ for $\pi_c=\Pi(\hat \eta^c)$ and $\nu_c=\sigma_c\tau_c^{-1}$ introduced in Sec.~\ref{sub:proof-of-the-theorem} is simply obtained by {\it contracting} the connected components of the nodal surface $\cS(\bsig, \btau)$ (not its irreducible components!) and those of each constellation $\hat\eta^c$ to points. This retains the information on which faces $F_j$ and $F'_j$ identified by $\Psi$ are in the same connected component of  $\cS(\bsig, \btau)$ on one hand, and which ones are in the same connected component of $\hat \eta^c$ on the other (this is why it only depends on the associated \emph{partitions}). The graph $ G \big[\Pi(\bsig, \btau) , \{\pi_c\}_c ; \{ \Pi(\nu_c) \}_c \big] $ for the example in Fig.~\ref{fig:big-nodal-surface} is the one in Fig.~\ref{fig:ex-incidence-graph}.

\begin{figure}[!h]
\centering
\includegraphics[scale=0.65]{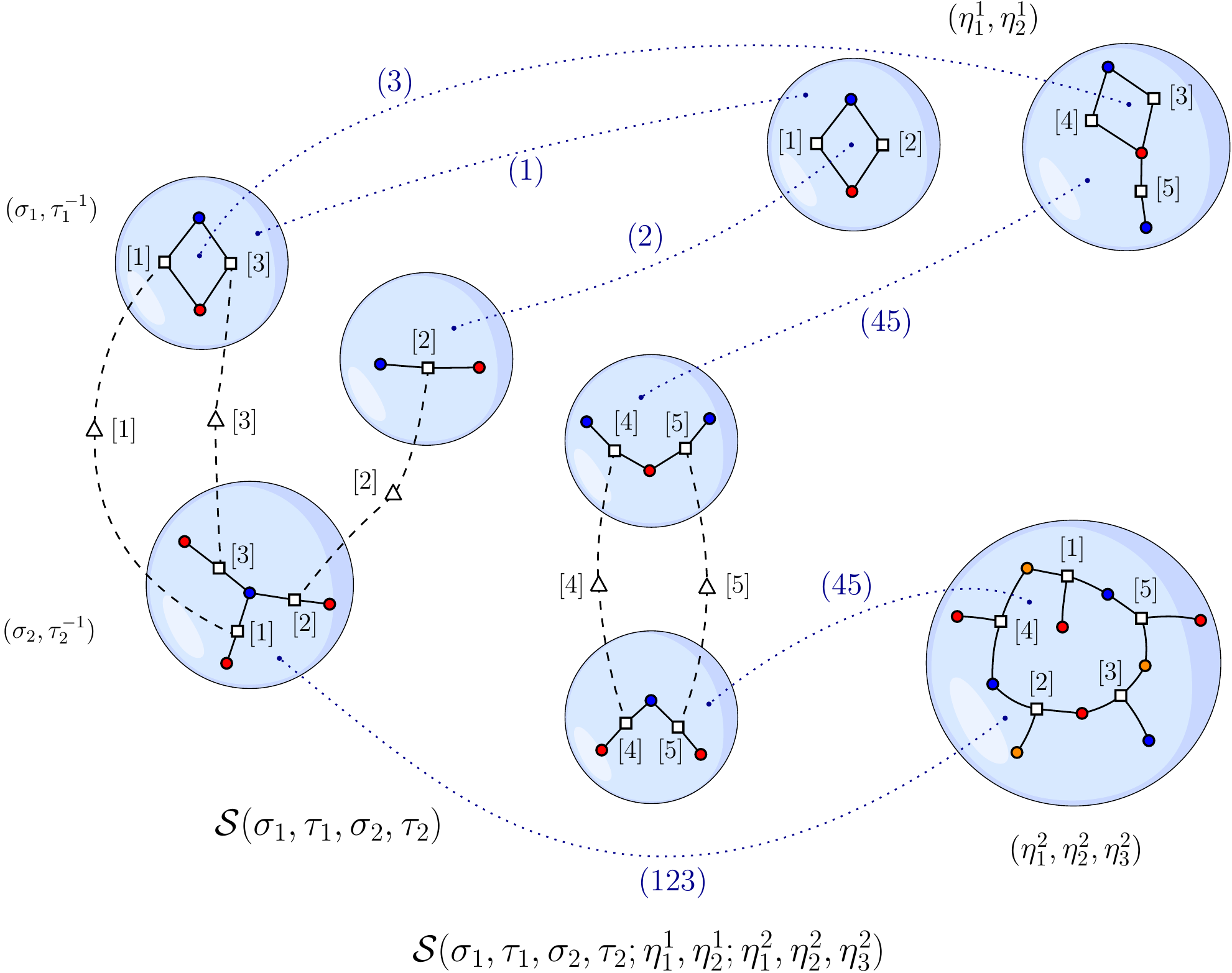} 
\caption{Graphical representation  of $\cS(\bsig,\btau;\bethat )$ for an example in $D=2$, $n=5$, where $\cS(\bsig, \btau)$ is as in Fig.~\ref{fig:nodal-surface-1}, and $\eta^1_{1}=(12)(34)(5)$, $\eta^1_{2}=(12)(345)$, $\eta^2_{1}=(15)(24)(3)$, $\eta^2_{2}=(1)(23)(4)(5)$, and $\eta^2_{3}=(14)(2)(35)$. The colors for the constellations  are blue for 1, red for 2, orange for 3. The dotted lines representing the pairwise identification of faces are labeled with the corresponding cycle.
}
\label{fig:big-nodal-surface}
\end{figure}

\begin{Lem}
\label{Lem:connectivity-big-nodal}
The number of connected components of  $\cS(\bsig,\btau,\bethat )$ is the number of connected components of $ G \big[\Pi(\bsig, \btau) , \{\Pi(\hat \eta_c)\}_c ; \{ \Pi(\nu_c) \}_c \big] $, which is also the number of transitivity classes of the group generated by all $\bsig,\btau,\bethat$, namely  $\lvert\Pi\bigl(\bsig,\btau,\bethat)\rvert$. Imposing transitivity in (C1) is imposing that $\cS(\bsig,\btau,\bethat )$ is connected. 
\end{Lem}

\begin{Lem}
\label{Lem:diff-arithm-genus-and-l}
With the notations \eqref{eq:colorwise-constellation-bis}, $l$ is related to the arithmetic genus $\mathcal G(\bsig, \btau, \bethat)$ of $\cS(\bsig,\btau,\bethat )$ by:
\be 
l=\sum_{c=1}^D\sum_{i=1}^{k_c} \lVert \eta^c_i \rVert = \sum_{c=1}^D\left[\#(\sigma_c) + \#(\tau_c)\right] + 2\mathcal G(\bsig, \btau, \bethat) - 2 - 2n(D-1).
\ee 
In particular, using \eqref{eq:relation-genus-ell-1}:
\be 
l - \ell(\bsig, \btau) = 2\bigl( \mathcal  G(\bsig, \btau, \bethat) - \mathcal  G(\bsig, \btau)\bigr).
\ee
Fixing $l$ in (C2) therefore amounts to fixing the arithmetic genus of $\cS(\bsig, \btau, \bethat)$ .
\end{Lem}
\proof From the definition \eqref{eq:nodal-genus}, $\mathcal G(\bsig, \btau, \bethat)$ reads:
$$
\mathcal G(\bsig, \btau, \bethat) = \sum_{c=1}^D \left(g(\sigma_c, \tau_c^{-1}) + g(\hat \eta^c) \right) + nD + \sum_{c=1}^D\#(\sigma_c\tau_c^{-1}) - n - \sum_{c=1}^D\left( \lvert \Pi(\sigma_c, \tau_c^{-1}) \rvert + 
\lvert \Pi(\hat \eta^c) \rvert \right) + 1. $$ Summing the Euler characteristics \eqref{eq:Euler-constellations} of $\hat \eta^c$:
$$
\sum_{c=1}^D \left( g(\hat \eta^c) -\lvert \Pi(\hat \eta^c) \rvert \right) =  \frac 1 2 (l - nD)- \frac 1 2   \sum_{c=1}^D \#(\sigma_c\tau_c^{-1})
$$  
The result follows using the Euler characteristics of $(\sigma_c, \tau_c^{-1})$ \eqref{eq:Eulersigmatau}.

 \qed

\ 




From these lemmas (for the point 1 of the proposition) as well as the results of Sec.~\ref{sub:Coverings} (for the point 2 of the proposition):

\begin{Prop} 
\label{prop:Topological-Expansion}
The $1/N$ expansion of the cumulant Weingarten functions in Thm.~\ref{thm:1Nexpansion-Weingarten-Cumulants} can be seen as a topological expansion:
\[
W_C^{(N)}[\bsig, \btau ]= (-1)^{\ell(\bsig, \btau)} N^{2\mathcal G(\bsig, \btau) - \ell(\bsig, \btau) - nD}  \sum_{\mathcal G \ge \mathcal G(\bsig, \btau)} \;\frac 1 {N^{2\mathcal G}} 
\,p_C [\bsig , \btau ; \ell(\bsig, \btau) + 2\left( \mathcal G - \mathcal G(\bsig, \btau)\right)]  \; , 
\]
where:
\[
(-1)^{\ell(\bsig, \btau)} p_C [\bsig , \btau ;l] =  \sum_{k\ge 0}\;(-1)^{k}\,m_C( \bsig , \btau ; l, k) \; , 
\]
and $m_C( \bsig , \btau ; \ell(\bsig, \btau) + 2\left( \mathcal G - \mathcal G(\bsig, \btau)\right), k)$ counts both:
\begin{enumerate}
\item The isomorphism classes of nodal surfaces $\cS(\bsig, \btau, \bethat)$, with $\bethat=(\hat \eta^1, \ldots, \hat \eta^D)$ as in \eqref{eq:colorwise-constellation-bis},
\item The nodal topological constellations encoded by $\bethat ' =(\hat {\eta'}^1, \ldots, \hat {\eta'}^D)$ where $\hat {\eta'}^c=(\sigma_c^{-1}, \hat \eta^c, \tau_c)$,
\end{enumerate}
and in both cases, the spaces are \textbf{connected} and of \textbf{fixed arithmetic genus} $\mathcal G \ge \mathcal G (\bsig, \btau)$, and so that the vertices of flavor $i$ of the $\hat \eta^c$ are not all of valency one. 
$p_C[\bsig, \btau; l]$ counts the subset of the spaces listed above for which the $\hat \eta^c$ consist of sequences of transpositions with weakly monotone maxima (this also translates on a condition on the flavored vertices).
\end{Prop}

\

The arithmetic genus $\mathcal  G(\bsig, \btau, \bethat) $, or equivalently the exponent $l= \sum_{c=1}^D\sum_{i_c=1}^{k_c} \lVert\eta^c_{i_c}\rVert$, can then be expressed as (Sec.~\ref{sub:proof-of-the-theorem}): 
\be
\label{eq:genera+Excess}
\mathcal  G(\bsig, \btau, \bethat) - \mathcal  G(\bsig, \btau) = \frac {l-\ell(\bsig, \btau)} 2  = \sum_{c=1}^D g(\hat \eta^c ) +  L   \big[\Pi(\bsig, \btau) , \{\pi_c\}_c ; \{ \Pi(\nu_c) \}_c \big]\;,
\ee
where $g(\hat \eta^c ) $ is the genus of the $k_c$-constellation $\hat\eta^c$ (the sum of genera of its connected components) and $L   \big[\Pi(\bsig, \btau) , \{\pi_c\}_c ; \{ \Pi(\nu_c) \}_c \big]$ the number of excess edges \eqref{eq:L} of $ G \big[\Pi(\bsig, \btau) , \{\pi_c\}_c ; \{ \Pi(\nu_c) \}_c \big] $:
\be 
\label{eq:Lbis}
L   \big[\Pi(\bsig, \btau) , \{\pi_c\}_c ; \{ \Pi(\nu_c) \}_c \big] =  \sum_c (|\Pi(\nu_c)| - \lvert\pi_c\rvert)-\lvert\Pi(\bsig, \btau)\rvert + 1. 
\ee.

This provides a better understanding on how to characterize and count the contributions to $m_C$ and $p_C$ in Prop.~\ref{prop:Topological-Expansion}: 
counting connected $\cS(\bsig, \btau, \bethat)$ of fixed arithmetic genus $\mathcal G$ amounts to counting those $\cS(\bsig, \btau, \bethat)$ for which the graph  $G\big[\dots \big]$ is connected and has excess $L$ between $0$ and $\mathcal G - \mathcal G(\bsig, \btau)$ while the genera $g(\hat \eta^c)$ sum up to $\mathcal G - \mathcal G(\bsig, \btau) - L$ (the other conditions in Prop.~\ref{prop:Topological-Expansion} must also be satisfied). In the example of Fig.~\ref{fig:big-nodal-surface}, the constellations are planar and $(l-\ell ) /2 = 2 = L $.  For instance, the following gives a prescription for generating all the spaces that contribute at leading order:
\begin{Prop}
The spaces $\cS(\bsig, \btau, \bethat)$ that correspond to the leading term in $N$ \eqref{eq:WC-asym}  for the expansion of the cumulant Weingarten functions $W_C^{(N)}[\bsig, \btau]$ are those of minimal arithmetic genus  $\mathcal  G(\bsig, \btau, \bethat) = \mathcal  G(\bsig, \btau)$, that is, so that  the $\hat \eta^c$ are all planar and so that $ G \big[\Pi(\bsig, \btau) , \{\pi_c\}_c ; \{ \Pi(\nu_c) \}_c \big]$ is a tree. 
\end{Prop}

From this picture the aim is to keep $\cS(\bsig, \btau)$ fixed and group the contributions of the different constellations $\bethat$ that lead to the same values of  $L$ and the same genera for the connected components of the $\hat \eta^c$. This is achieved in the last subsection.

\subsubsection{A simpler kind of nodal surfaces} 
\label{sec:nodal-pC}

Both for $m_C$ and $p_C$, one fixes $\bsig, \btau$ but sums over the proper $\bethat$ (denoted by $\brho$ or $\bhatmu$) satisfying a number of assumptions. 
In order to understand this geometrically, one may therefore, from  $\cS(\bsig,\btau;\brho )$ introduced in the previous subsection, contract to points the connected components of the constellations $\hat \eta^c$ but keep the nodal topological constellation $\cS(\bsig, \btau)$ (that is go only half-way to build the graph $G\big[\Pi(\bsig, \btau) , \{\pi_c\}_c ; \{ \Pi(\nu_c) \}_c \big]$). The result is a new kind of object $\cS(\bsig, \btau ; \{ \pi_c\}_c)$ (since only the information on the connected components of the constellations is retained, they have been replaced in the argument by the corresponding partitions on $\{1,\ldots, n\}$). It can be obtained directly from  $\cS(\bsig, \btau)$ by adding nodal points between the faces of $\cS(\bsig, \btau)$ corresponding to the blocs of $\{\pi_c\}_c$. Let us introduce this object more formally. 

\ 

We fix  $\bsig, \btau\in \bS_n$ as well as $\mathcal G\ge 0$, and for $c\in \{1,\ldots, D\}$, we let $\pi_c\ge \Pi(\nu_c)$ be a partition of the disjoint cycles of $\nu_c=\sigma_c\tau_c^{-1}$ 
subject to the conditions:
\begin{enumerate}[label=(C'\arabic*)]
\item $\Pi(\bsig, \btau)\vee\pi_1\vee\ldots\vee\pi_D=1_n$,
\item $L   \big[\Pi(\bsig, \btau) , \{\pi_c\}_c ; \{ \Pi(\nu_c) \}_c \big]  = \mathcal G - \mathcal G (\bsig,\btau)$.
\end{enumerate}

For every nodal embedded graph $\Gamma^\bullet \subset X^\bullet$ in the isomorphism class $\cS(\bsig, \btau)$, $\{\pi_c\}_c$ provides a partition of the faces $F_1, \ldots, F_{nk-l}$ of $\Gamma^\bullet \subset X^\bullet$ (the connected components of the complement of the graph $\Gamma^\bullet$ in the nodal surface $ X^\bullet$).  We may see the blocks in this partition as a new kind of node: we choose for each $j\in\{1,\ldots, nk-l\}$ a point $v_j$ in the interior of $F_j$, and we see each block $B_c$ of $\pi_c$ for each $c$ as  a node $P(B_c) = \{v_j \mid F_j \in B_c\}$. 
We then consider the nodal surface $X^\bullet  / \sqcup_c \{P(B_c)\}_{B_c\in\pi_c}$  (together with the nodal graph $\Gamma^\bullet$ embedded on $X^\bullet$).

We then denote by $\cS(\bsig, \btau ; \{ \pi_c\}_c)$ the isomorphism class of such objects, where by isomorphisms we mean the homeomorphisms of nodal surfaces that induce an isomorphism of nodal embedded graph on $\Gamma^\bullet\subset X^\bullet$ and that preserve the incidence between the nodes $P(B_c)$ and the faces $F_j$, in the sense that if a node $P(B_c)$ belongs to the interiors of the faces $F_j\in B_c$, then the image of the node also belongs to  the interior of the images of the faces. We let:
\be 
\mathfrak{S}_{\mathcal G}(\bsig, \btau)=\Bigl\{\cS(\bsig, \btau ; \{ \pi_c\}_c) \textrm{ such that } \{ \pi_c\}_c \textrm{ satisfy (C'1) and (C'2)}\Bigr\},
\ee
and more generally, we call the $\cS(\bsig, \btau ; \{ \pi_c\}_c)$ the \emph{foldings of $\cS(\bsig, \btau)$. }

An example is shown in Fig.~\ref{fig:big-nodal-surface-2}, where now the nodal points in the interiors of the faces are represented by dotted edges that meet at star-vertices  labelled by the blocks of the partitions $\pi_c$.

\begin{figure}[!h]
\centering
\includegraphics[scale=0.7]{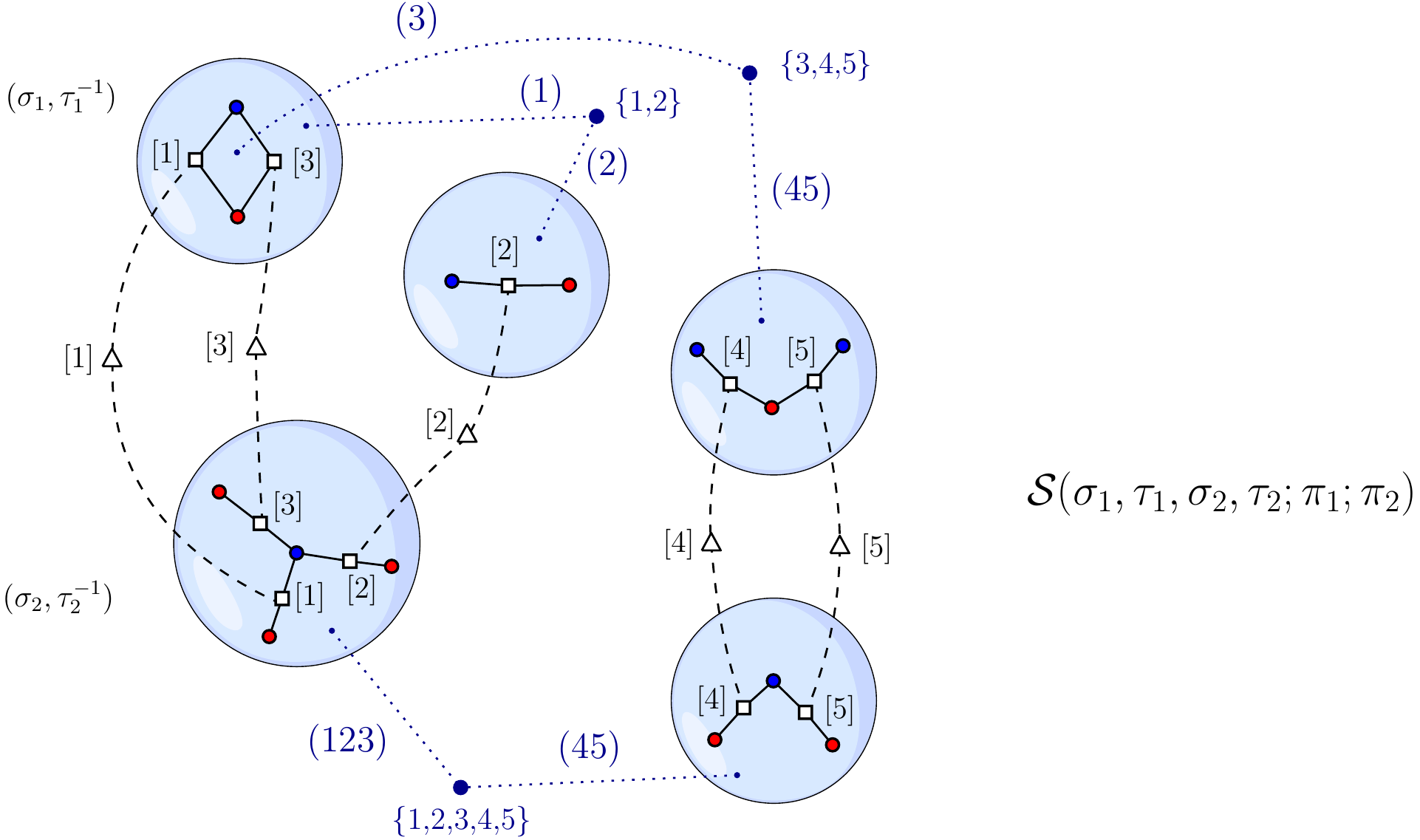} 
\caption{The nodal surface $\cS(\bsig, \btau ; \pi_1,\pi_2)$ for the example of Fig.~\ref{fig:big-nodal-surface}, where $\cS(\bsig, \btau)$ is as in Fig.~\ref{fig:nodal-surface-1} and with $\pi_1=\{1,2,3\}\{4,5\}$ and $\pi_2=\{1,2,3, 4,5\}$.
}
\label{fig:big-nodal-surface-2}
\end{figure}

As for Lemma~\ref{Lem:connectivity-big-nodal}, the following follows directly from the fact that $ G \big[\Pi(\bsig, \btau) , \{\pi_c\}_c ; \{ \Pi(\nu_c) \}_c \big] $ is obtained by contracting the connected components of $\cS(\bsig, \btau)$ to points:

\begin{Lem}
The number of connected components of  $\cS(\bsig,\btau; \{ \pi_c\}_c )$ is $\lvert \Pi(\bsig, \btau)\vee\pi_1\vee\ldots\vee\pi_D \rvert$.
\end{Lem}

However now, in comparison to Lemma~\ref{Lem:diff-arithm-genus-and-l}, the information on the genera of the connected components of the $\hat \eta^c$ has been lost:

\begin{Lem}
The arithmetic genus of a connected folding $\cS(\bsig,\btau; \{ \pi_c\}_c )$ can be expressed as:
\be
\mathcal G (\bsig,\btau; \{ \pi_c\}_c ) =  \mathcal G (\bsig,\btau) + L   \big[\Pi(\bsig, \btau) , \{\pi_c\}_c ; \{ \Pi(\nu_c) \}_c \big] .
\ee
\end{Lem}
\proof From the definition \eqref{eq:nodal-genus}, $\mathcal G(\bsig, \btau; \{ \pi_c\}_c )$ reads:
$$
\mathcal G(\bsig, \btau, \bethat) = \sum_{c=1}^D g(\sigma_c, \tau_c^{-1}) + nD + \sum_{c=1}^D\#(\sigma_c\tau_c^{-1}) - n - \sum_{c=1}^D\left( \lvert \Pi(\sigma_c, \tau_c^{-1}) \rvert + \pi_c\right) + 1. $$ The result follows using \eqref{eq:nodal-genus2} for $\cS(\bsig, \btau)$ as well as \eqref{eq:Lbis}. 

\qed

\

We can therefore translate the conditions (C'1) and (C'2) geometrically:

\begin{Lem}
$\mathfrak{S}_{\mathcal G} (\bsig, \btau)$ is the set of connected foldings of $\cS(\bsig, \btau)$ whose arithmetic genus is $\mathcal G$.
\end{Lem}

\

We may express both $m_C$ and $p_C$ in terms of connected  foldings of $\cS(\bsig, \btau)$ of bounded arithmetic genus. While in Prop.~\ref{prop:Topological-Expansion}, the $\cS(\bsig,\btau;\brho )$   were counted with a weight one, now the $\cS(\bsig, \btau ; \{ \pi_c\}_c)$ must be counted with a non-trivial weight that takes into account all the different choices of $\bethat$ satisfying the conditions in Prop.~\ref{prop:Topological-Expansion} and that lead to the same $\cS(\bsig, \btau ; \{ \pi_c\}_c)$. 
 For $m_C( \bsig , \btau ; l, k)$, from \eqref{eq:M}, for each color $c$ this weight is the number $M(\pi_c, \nu_c ; l_c, k_c)$  of proper $k_c$-constellations which  respect to the partition $\pi_c$, that is $\Pi(\hat\eta^c) = \pi_c$. However, as already mentioned in Sec.~\ref{sub:proof-of-the-theorem}, this quantity does not factor over the connected components of $\hat\eta^c$, since a permutation $\eta^c_i$ may still be different from the identity but reduce to the identity on a subset of $\{1,\ldots, n\}$ corresponding to a connected component of $\hat\eta^c$ (said graphically, the vertices of flavor $i$ may all be the identity on a connected component of $\hat\eta^c$ without it being the case for all connected components of $\hat\eta^c$). This means that $M(\pi_c, \nu_c ; l_c, k_c)$ cannot be expressed as a product of weights associated to some of the vertices of $\cS(\bsig, \btau ; \{ \pi_c\}_c)$.
 
 \
 
 On the other hand, for $p_C[ \bsig , \btau ; l]$, one has factorization over the blocks of $\pi_c$, leading to the following geometrical interpretation of the formula \eqref{formula-gen-l} of Thm.~\ref{thm:asympt-cum-weing}  in terms of the nodal surfaces $\cS(\bsig, \btau ; \{ \pi_c\}_c)$, with non-trivial combinatorial weights. We let $V^c_\bullet(\cS)$ the set of nodes of a folding $\cS=\cS(\bsig, \btau ; \{ \pi_c\}_c)$ of $\cS(\bsig, \btau)$ that correspond to the blocks of $\pi_c$, called \emph{nodes of color $c$}, and by $V_\bullet(\cS) = \sqcup_c V^c_\bullet(\cS)$ the set of \emph{colored nodes}. For $v\in V^c_\bullet(\cS)$, we also let ${\nu}_{c \mid v}$ be the restriction of $\nu_c=\sigma_c\tau_c^{-1}$ to the block $B_c\in \pi_c$ corresponding to $v$, and $c({\nu}_{c \mid v})$ is its cycle type, which is a partition of $\lvert B_c \rvert$. We recall that for $\alpha=c({\nu}_{c \mid v})$, $\lvert C_\alpha \rvert = \frac{\lvert B_c \rvert !}{ \prod_{p\ge 1 p^{d_p(\alpha)} d_p(\alpha)!}}$, where $d_p(\alpha)$ is the number of parts of $\alpha$ of size $p$. 
 \begin{Prop}
 With these notations, $p_C[ \bsig , \btau ; l]$ can be expressed as a sum of connected foldings of $\cS(\bsig, \btau)$ of bounded arithmetic genus, whose  colored nodes are weighted by monotone single Hurwitz numbers:
 \be
 p_C[ \bsig , \btau ; l] = \sum_{\mathcal G = \mathcal G (\bsig, \btau)} ^ {\frac 1 2 [l-\ell(\bsig, \btau )]} \sum_{\mathcal S \in \mathfrak{S}_{\mathcal G} (\bsig, \btau)} \sum_{\substack{{\{g_v\}_{v\in V_\bullet(\cS)}\ge 0}\\{\sum_{v\in V_\bullet(\cS)} g_v = \frac 1 2 [l-\ell(\bsig, \btau )] +  \mathcal G (\bsig, \btau) - \mathcal G }}} \prod_{v\in V^c_\bullet(\cS)} \frac {\vec H_{g_v}(c({\nu}_{c \mid v}))} {\lvert C_{c({\nu}_{c \mid v})}\rvert}\;.
\ee 
The topological expansion of the cumulant Weingarten functions in Prop.~\ref{prop:Topological-Expansion} can therefore also be re-expressed as a topological expansion over connected foldings of  $\cS(\bsig, \btau)$ of fixed arithmetic genus.
 \end{Prop}


With this interpretation, generating all the contributions to  $ p_C[ \bsig , \btau ; l]$ is quite simple: fixing $\bsig, \btau$ and $l$, one sums over the excess $L$ between 0 and 
 $(l-\ell)/2$, and over all possible ways to add nodes of color $c$ (represented by star-vertices of color $c$) for every color $c\in\{1,\ldots, D\}$ between all the faces of $\cS(\bsig, \btau)$, so that the resulting (class of) nodal surface is connected and the graph obtained when contracting the connected components of $\cS(\bsig, \btau)$ to points has $L$ excess-edges. This generates all the  foldings of $\cS(\bsig, \btau)$ of arithmetic genus $\mathcal G(\bsig, \btau) + L$.  
 For each such folding $\cS=\cS(\bsig, \btau ; \{ \pi_c\}_c)$, each node of color $c$ corresponds to a block $B_c$ of $\pi_c$. One then distributes the total genus $\tfrac {l-\ell } 2  -  L  $ (see \eqref{eq:genera+Excess}) among all the nodes of color $c$, and each such node is endowed with a weight $\gamma_{l(g_{B_c}) }(\nu_{c|B_c})$, that precisely takes into account the contributions of all the connected constellations of genus $g_{B_c}$ corresponding to the connected component of $\hat \eta^c$ for every $\cS(\bsig,\btau,\bethat )$ that contracts to $\cS$. 
 From \eqref{eq:Simple-Hurw-gamma} this factor $\gamma_{l(g_{B_c}) }(\nu_{c|B_c})$ is proportional to a genus-$g_{B_c}$ single monotone Hurwitz number (the signs combine in the overall $(-1)^l$ factor in \eqref{formula-gen-l}).
 
 The simplified version of this geometrical picture corresponding to $D=1$ and minimal $l=\ell$ has been introduced in \cite{Zub-ZJ}.

\section*{Acknowledgements}
B.C. was supported by JSPS KAKENHI 17K18734
and 17H04823. 
R.G. is supported by the European Research Council (ERC) under the European Union’s Horizon 2020 research and innovation program (grant agreement No818066) and by Deutsche Forschungsgemeinschaft (DFG, German Research Foundation) under Germany's Excellence Strategy  EXC-2181/1 - 390900948 (the Heidelberg STRUCTURES Cluster of Excellence).
L.L.~acknowledges support of the START-UP 2018 programme with project number 740.018.017, which is financed by the Dutch Research Council (NWO). This project has also received funding from the European Research Council (ERC) under the European Union’s Horizon 2020 research and innovation programme (grant agreement No. ERC-2016-STG 716083, ”CombiTop”). L.L.~thanks JSPS and Kyoto University, where the discussions at the origin of this project took place.
L.L.~thanks G.~Chapuy for useful references on constellations.
The authors would like to thank Jonathan Novak for insightful comments on a preliminary version.

\end{document}